\documentclass{article}
\usepackage{amsmath}
\usepackage{amssymb, amscd, amsthm, amsfonts, latexsym}
\usepackage{enumerate}
\usepackage[dvips]{graphicx,color}

\usepackage{latexsym}
\newcommand{\no}{\noindent}
\newtheorem{thm}{Theorem}[section]
\newtheorem{prop}[thm]{Proposition}
\newtheorem{cor}[thm]{Corollary}
\newtheorem{rem}[thm]{Remark}

\newtheorem{ex}[thm]{Example}
\newtheorem{lem}[thm]{Lemma}

\newtheorem{notation}[thm]{Notation}
\numberwithin{equation}{section}

\newcommand{\R}{\mathbb{R}}

\newcommand{\Z}{\mathbb{Z}}
\newcommand{\N}{\mathbb{N}}
\newcommand{\ds}{\displaystyle}

\newcommand{\sm}{\setminus}
\newcommand{\pd}{\partial}
\newcommand{\al}{\alpha}

\newcommand{\Ga}{\Gamma}
\newcommand{\de}{\delta}
\newcommand{\De}{\Delta}
\newcommand{\ep}{\varepsilon}

\newcommand{\f}{\varphi}
\newcommand{\ome}{\omega}

\renewcommand{\(}{\left(}
\renewcommand{\)}{\right)}
\renewcommand{\lvert}{\left\vert}
\renewcommand{\rvert}{\right\vert}
\DeclareMathOperator*{\esssup}{ess\,sup}

\DeclareMathOperator{\supp}{supp}

\DeclareMathOperator{\dist}{dist}

\def\vect#1{\mbox{\boldmath $#1$}}

\begin{document}

\title{\bf Movement of time-delayed hot spots in Euclidean space}

\author{\bf Shigehiro Sakata and Yuta Wakasugi}

\date{\today}

\maketitle

\begin{abstract} 
We investigate the shape of the solution of
the Cauchy problem for the damped wave equation.
In particular, we study the existence, location and number of spatial maximizers of the solution.

Studying the shape of the solution of the damped wave equation,
we prepare a decomposed form of the solution into the heat part and the wave part.
Moreover, as its another application,
we give $L^p$-$L^q$ estimates of the solution.\\   

\no{\it Keywords and phrases}.
Damped wave equation,
diffusion phenomenon, hot spot, $L^p$-$L^q$ estimate.\\
\no 2010 {\it Mathematics Subject Classification}:
35L15, 35B38, 35C15, 35B40, 35K05.
\end{abstract}

\section{Introduction}
Let $f$ and $g$ be real-valued smooth functions defined on $\R^n$.
We consider the {\it damped wave equation} with initial data $(f,g)$,
\begin{equation}\label{DWfg}
\begin{cases}
\ds \( \frac{\pd^2}{\pd t^2} -\De + \frac{\pd}{\pd t} \) u(x,t) =0, &x \in \R^n ,\ t>0,\\
\ds \( u, \frac{\pd u}{\pd t} \) (x,0)=(f, g)(x), &x \in \R^n .
\end{cases}
\end{equation}

In the one-dimensional case,
the damped wave equation is known as the {\it telegrapher's equation} introduced by Oliver Heaviside
and describes the current and voltage in an electrical circuit with
resistance and inductance.
More generally, the equation \eqref{DWfg} is a model of
the propagation of the wave with friction or resistance.

The damped wave equation is also known as the {\it hyperbolic heat conduction equation}
introduced in [Cat, Ch, L, MF, V].
The classical model of the heat equation admits the infinite speed of the propagation of heat conduction, which is physically inadmissible.
Therefore, the equation \eqref{DWfg} was introduced to modify
the model of heat conduction with finite speed of the propagation.
In order to derive the equation \eqref{DWfg} as the hyperbolic heat conduction equation, along Li's framework in [L], 
let us consider the one-dimensional case as below:
Let $v(x,t)$ be the temperature at a point $x \in \R$ and
at time $t$; Let $q(x,t)$ be the heat flux at a point $x \in \R$ and at time $t$;
Then, the heat balance law implies
\begin{equation}\label{heat_balance}
	\frac{\pd v}{\pd t}(x,t) + \frac{\pd q}{\pd x} (x,t) =0,\ x \in \R ,\ t>0;
\end{equation}
From the time-delayed Fourier's law with a small enough positive parameter $\tau$
\begin{equation}\label{Fourier_delay}
	q(x,t+\tau ) \approx q(x,t) + \tau \frac{\pd q}{\pd t}(x,t)
	= -\frac{\pd v}{\pd x} (x,t) ,\ x \in \R ,\ t>0,
\end{equation}
instead of the usual Fourier's law
\begin{equation}\label{Fourier}
q(x,t)=-\frac{\pd v}{\pd x}(x,t) ,\ x \in \R ,\ t>0,
\end{equation}
we get the damped wave equation
\begin{equation}\label{dwv}
	\( \tau \frac{\pd^2}{\pd t^2} -\frac{\pd^2}{\pd x^2}
	+ \frac{\pd}{\pd t} \) v(x,t) =0 ,\ x\in \R ,\ t>0,
\end{equation}
with delay time $\tau$;
The additional term $\tau \pd^2 v / \pd t^2$
brings the finite propagation speed property to the equation (\ref{dwv});
By the scale transformation 
\begin{equation}\label{scale_trans}
	u(x,t) = \frac{1}{\tau} v\( \sqrt{\tau}x ,\tau t \) ,
\end{equation}
we obtain
\begin{equation}\label{dwu}
	\( \frac{\pd^2}{\pd t^2} -\frac{\pd^2}{\pd x^2}
		+ \frac{\pd}{\pd t} \) u(x,t) =0,\ x \in \R ,\ t>0.
\end{equation}

From such a background, as $t$ goes to infinity in (\ref{dwu})
(corresponding to the case where $\tau$ tends to $0^+$ in (\ref{dwv})),
it is expected that the solution of the damped wave equation
approaches to that of the (usual) heat equation, which is called the
{\it diffusion phenomenon} and has been studied by many researchers
([FG, HO, MN, Nar, Nis, YM]).

In this paper, 
we investigate the relation between the damped wave and heat equations in view of the study on the shapes of the solutions. Precisely, we give correspondence to
Chavel and Karp's results in [CK].
Let us review their works as below:
Let $P_n(t) \phi (x)$ be the unique bounded solution of
the Cauchy problem for the heat equation with bounded initial datum $\phi$, that is,
\begin{equation}\label{P}
P_n(t) \phi (x)= \frac{1}{\( 4\pi t\)^{n/2}}
\int_{\R^n} \exp \( -\frac{r^2}{4t} \) \phi (y) dy,\ x \in \R^n ,\ t>0,\ r=\lvert x-y \rvert ;
\end{equation}
When $\phi$ is a non-zero non-negative bounded function with compact support, they studied the behavior of the set of {\it hot spots}
\begin{equation}\label{H}
H_\phi (t) = \left\{ x \in \R^n \lvert P_n(t)\phi (x) = \max_{\xi \in \R^n} P_n (t) \phi (\xi ) \right\} \right. ;
\end{equation}
They showed that hot spots exist at each time $t$,
that all of them are contained in the convex hull of the support of $\phi$ for any time $t$,
and that the set $H_\phi (t)$ converges to the one-point set of the centroid
(the center of mass) of $\phi$ as $t$ goes to infinity;
Furthermore, calculating the Hessian of $P_n(t)\phi$,
in [JS], Jimbo and Sakaguchi indicated that the set of hot spots $H_\phi (t)$
consists of one point after a large time $t$.

There are many results on the study of hot spots besides [CK]. As examples, we introduce [JS, I1, I2, IK, FI1, FI2, S] as below: In [JS], Jimbo and Sakaguchi studied the large time behavior of hot spots in unbounded domains. For example, they considered the exterior domain of a ball with a radially symmetric initial datum. In this case, the explicit representation of solutions like \eqref{P} was still useful; In [I1, I2], Ishige studied the large time behavior of hot spots in the exterior domain of a ball without the radially symmetric assumption of [JS] by using the self-similar transformation and the eigenfunction expansion: In [IK], the similar approach to [I1, I2] was also applicable for the heat equation with a potential; In [FI1, FI2], Fujishima and Ishige studied the large time behavior of hot spots in $\R^n$ without the non-negativity of the initial datum $\phi$ and applied their investigation to the blow-up set of a semi-linear heat equation; In [S], the first author generalized the results in [CK, JS] in terms of a potential with a radially symmetric kernel.

In view of the diffusion phenomena,
it is expected that spatial maximizers of the solution of (\ref{DWfg})
have similar properties to hot spots shown in [CK, JS, S].
To this aim, 
when $f$ and $g$ are compactly supported,  and when $h:= f+g$ is non-zero and non-negative, 
we study the behavior of the set of {\it time-delayed hot spots}
\begin{equation}\label{dH} 
\mathcal{H}(t) = \left\{ x \in \R^n \lvert u(x,t)
= \max_{\xi \in \R^n} u(\xi, t) \right\} \right. .
\end{equation} 
Precisely, we show the following properties:
\begin{enumerate}
\item[(1)] After a large enough time, the set
$\mathcal{H}(t)$ is contained in the convex hull of the support of $h$.
Furthermore, for some small time, we give some examples of $(f,g)$ such that the set
$\mathcal{H}(t)$ escapes from the convex hull of the support of $h$
(Theorem \ref{movement} (1), Examples \ref{ex1}, \ref{ex2}, \ref{ex3} and \ref{ex4}).
\item[(2)] The set $\mathcal{H}(t)$ converges to
the one-point set of the centroid of $h$ as $t$ goes to infinity
(Theorem \ref{movement} (2)).
\item[(3)] After a large enough time, the set $\mathcal{H}(t)$ consists of one point
(Proposition \ref{uniqueness}).
\end{enumerate}
In order to understand the meaning of the above statements, we, for example, consider the case where $f$ is non-zero and non-negative, the maximum value of $f$ is greater than that of $h$, and the supports of $f$ and $h$ are separated. Since $u(x ,0) =f(x)$, for any sufficiently small $t>0$, all of the time-delayed hot spots are ``close'' to maximum points of $f$, that is, they are contained in the support of $f$ and not contained in the support of $h$. Roughly speaking, our main results claim that time-delayed hot spots move from the set of maximum points of $f$ to the centroid of $h$.

The above properties of time-delayed hot spots are due to the decomposition of the solution operator (fundamental solution)
$S_n(t)$ into the heat part and the wave part as
\begin{equation}\label{decomp}
S_n(t)g(x) = J_n(t) g(x) + e^{-t/2}\vect{W}_n(t)g(x).
\end{equation}
Here, $J_n(t)g(x)$ and $\vect{W}_n(t) g(x)$
are suitable functions behaving like as $P_n(t)g(x)$ and the solution of
the free wave equation with initial datum $(0,g)$,
respectively. 
The decomposition \eqref{decomp}
was firstly discovered by Nishihara in [Nis]
in the three-dimensional case
and so-called the
{\it Nishihara decomposition}.
In this paper, we give the generalization of the Nishihara decomposition
in higher dimensional cases to study the behavior of time-delayed hot spots.
Moreover, as its another application,
we slightly improve Narazaki's $L^p$-$L^q$ estimates given in [Nar].
As a by-product of the
$L^p$-$L^q$ estimates, we obtain
\begin{equation}\label{LpLq}
	\| u(\cdot ,t) - P_n(t)h \|_{L^{\infty}}
		\le C t^{-n/2-1}
		\( \| f \|_{L^1} + \|g\|_{L^1} + \| f \|_{W^{*,\infty}} + \|g\|_{W^{*,\infty}} \) .
\end{equation}
One may consider that the large time behavior of time-delayed hot spots can be easily investigated by combining the results in [CK] and the estimate \eqref{LpLq}. But the expectation is incorrect. This is because the difference between the values of $P_n(t)h$ in the convex hull of support of $h$ and its outside can have the order worse than $t^{-n/2-1}$. Hence the above estimate cannot exclude the possibility that time delayed hot spots escape from the convex hull of support of $h$ by the effect of the wave part. Therefore, we should obtain more precise information about the solution of the damped wave equation.\\

This paper is organized as follows. In Section 2, we give a representation of the solution of \eqref{DWfg} and its Nishihara decomposition.
In Section 3, we give preliminary estimates for our investigation. In Section 4, we investigate the movement of time-delayed hot spots. In Section 5, we give $L^p$-$L^q$ estimates of the difference between the damped wave and the heat by using the Nishihara decomposition described in Section 2.\\

\no{\bf Notation.} For the end of this section, we explain our notation.
\begin{itemize}
\item The letter $C$ indicates the generic constant which may change from line to line.
\item We denote the usual $L^p$ norm by $\|\cdot\|_{L^p}$, that is,
\begin{equation}
\| \phi \|_{L^p}=
\begin{cases}
\ds \( \int_{\R^n} \lvert \phi (x) \rvert^p dx\)^{1/p} &( 1\leq p <\infty ),\\
\ds \esssup_{x \in \R^n} \lvert \phi (x) \rvert &(p=\infty) .
\end{cases}
\end{equation}
\item For a natural number $\ell$, we denote the Sobolev norm by $\| \cdot \|_{W^{\ell ,\infty}}$, that is,
\begin{equation}
	\| \phi \|_{W^{\ell ,\infty}} =\sum_{|\al |\leq \ell} \| \pd^\al \phi \|_{L^\infty} .
\end{equation}
\item Let $B^n$ and $S^{n-1}$ be the $n$-dimensional unit closed ball and the $(n-1)$-dimensional unit sphere, respectively.
\item For real numbers $a$ and $b$, and for two sets $X$ and $Y$ in $\R^n$, we use the notation (Minkowski sum) $aX+bY = \left\{ ax+by \lvert x \in X,\ y\in Y \right\} \right.$. In particular, we write $B_t^n(x) =tB^n +\{ x \}$ and $S_t^{n-1}(x) = tS^{n-1}+ \{ x \}$, that is, the $n$-dimensional closed ball with radius $t$ centered at $x$ and the $(n-1)$-dimensional sphere with radius $t$ centered at $x$, respectively.
\item Let us denote by $CS(\phi )$ the convex hull of the support of a function $\phi$.
\item Let $\sigma_n$ denote the $n$-dimensional Lebesgue surface measure.
\item We understand that the letter $r$ is always used for $r=\lvert x-y \rvert$.
\end{itemize}

\no{\bf Acknowledgements.} The authors would like to express their deep gratitude to
Professor Tatsuo Nishitani and Professor Jun O'Hara for giving them valuable comments.

\section{Decomposition of the solution} 
\subsection{Solution formula of the damped wave equation}
In this subsection, we prepare the explicit form of the solution of the Cauchy problem \eqref{DWfg}. For a smooth function $g$, let us denote by $S_n(t)g$ the solution of the Cauchy problem 
\begin{equation}\label{DWg}
\begin{cases}
\ds \( \frac{\pd^2}{\pd t^2} -\De + \frac{\pd}{\pd t} \) u(x,t) =0, &x \in \R^n ,\ t>0,\\
\ds \( u, \frac{\pd u}{\pd t} \) (x,0)=(0, g)(x), &x \in \R^n.
\end{cases}
\end{equation}
The symbol $S_n(t)$ is called the solution operator of \eqref{DWg}.

Let $I_{\nu}(s)$ be the modified Bessel function of order $\nu$,
\begin{equation}
I_{\nu}(s)
=\sum_{j=0}^{\infty}\frac{1}{j!\Ga (j+\nu+1)}\(\frac{s}{2}\)^{2j+\nu}.
\end{equation}
Put
\begin{equation}\label{cn}
c_n = 
\begin{cases}
1 &(n =1),\\
\( (n-2)!! \sigma_{n-1}\( S^{n-1} \) \)^{-1} = 2^{-(n+1)/2} \pi^{-(n-1)/2} &(n \in 2\N +1),\\
\( (n-1)!! \sigma_{n}\( S^n \) \)^{-1} = 2^{-(n+2)/2} \pi^{-n/2} &(n\in 2\N),
\end{cases}
\end{equation}
where $(2\ell -1) !! = (2\ell -1 ) \cdot (2\ell -3 ) \cdot \cdots \cdot 3 \cdot 1$ and $(2\ell )!! = 2 \ell \cdot (2\ell -2) \cdot \cdots \cdot 4 \cdot 2$.

\begin{prop}\label{propsol}
Let $g$ be a smooth function. For any natural number $n$, 
we have
\[
S_n(t)g(x)=
\begin{cases}
	\ds \frac{e^{-t/2}}{2}\int_{x-t}^{x+t}
		I_0\(\frac{1}{2}\sqrt{t^2-r^2}\)g(y)dy &(n=1),\\
	\ds c_n e^{-t/2}
	\(\frac{1}{t}\frac{\pd}{\pd t}\)^{(n-1)/2}
	\int_{B_t^n(x)}I_0\(\frac{1}{2}\sqrt{t^2-r^2}\)g(y)dy &(n\in 2 \N +1 ),\\
	\ds 2c_n e^{-t/2}
	\(\frac{1}{t}\frac{\pd}{\pd t}\)^{(n-2)/2}
	\int_{B_t^n(x)}\frac{\cosh \( \frac{1}{2}\sqrt{t^2-r^2} \)}{\sqrt{t^2-r^2}}
	g(y)dy &(n\in 2\N ).
\end{cases}
\]
\end{prop}

\begin{proof}
The argument is due to the method of descent described in [CH].

It is well known that the solution of the free wave equation 
\[
\begin{cases}
	\ds \( \frac{\pd^2}{\pd t^2}-\De \)
		w (x,t) =0,&(x,t)\in \R^n \times (0,\infty),\\
	\ds \( w, \frac{\pd w}{\pd t} \) (x,0)=(0,g)(x),&x\in\R^n 
\end{cases}
\]
is given by
\[
W_n(t)g(x)=
\begin{cases}
	\ds \frac{1}{2}\int_{x-t}^{x+t}g(y)dy &(n=1),\\
	\ds c_n
	\(\frac{1}{t}\frac{\pd}{\pd t}\)^{(n-3)/2}
	\(\frac{1}{t} \int_{S_t^{n-1}(x)} g(y) d\sigma_{n-1}(y) \) &(n\in 2\N +1),\\
	\ds 2c_n
	\(\frac{1}{t}\frac{\pd}{\pd t}\)^{(n-2)/2}
	\(\int_{B_t^n(x)} \frac{1}{\sqrt{t^2-r^2}}g(y)dy\) &(n \in 2\N ).
\end{cases}
\]

For
a point $x=(x_1,\ldots,x_n) \in \R^n$ and a number $\xi \in \R$,
we write
$z=(x,\xi) \in \R^{n+1}$.
Let $u(x,t)$ be the solution of \eqref{DWfg}, and $w(z,t) = \exp ( (\xi +t)/2) u(x,t)$.
Then,
$w(z,t)$
satisfies the
$(n+1)$-dimensional free wave equation with initial datum
$(w , \pd w/\pd t)(z,0)=( 0, e^{\xi/2}g(x) )$.

We first consider the one-dimensional case. Using the fact \eqref{Bessel_1}, we have
\begin{align*}
	u(x,t) 
	&=e^{-(\xi +t)/2} w(z,t) \\
	&=\frac{e^{-t/2}}{2\pi} \int_{B_t^2(z)} \frac{e^{\(y_2 -\xi \)/2}g\( y_1 \)}{\sqrt{t^2 -\lvert x- y_1 \rvert^2 -\lvert \xi - y_2 \rvert^2}} dy\\
	&=\frac{e^{-t/2}}{2\pi} \int_{x-t}^{x+t} g\( y_1 \) \( \int_{-\sqrt{t^2 -\lvert x-y_1 \rvert^2}}^{\sqrt{t^2 -\lvert x-y_1 \rvert^2}} \frac{e^{\(y_2 -\xi \)/2}}{\sqrt{t^2 -\lvert x- y_1 \rvert^2 -\lvert \xi - y_2 \rvert^2}} dy_2\) dy_1\\
	&=\frac{e^{-t/2}}{2} \int_{x-t}^{x+t} I_0 \( \frac{1}{2} \sqrt{t^2 -\lvert x -y_1 \rvert^2} \) g\( y_1 \) dy_1 .
\end{align*}

When $n$ is an odd number greater than one, as we have
\begin{align*}
	&\int_{B_t^{n+1}(z)}
	\frac{e^{\(y_{n+1} -\xi \)/2}g\( y'\)}{\sqrt{t^2 -\lvert x -y' \rvert^2 -\lvert \xi - y_{n+1} \rvert^2 }}dy \\
	& =\int_{B_t^n(x)} g\( y' \)
		\( \int_{-\sqrt{t^2-\lvert x-y'\rvert^2}}^{\sqrt{t^2-\lvert x-y'\rvert^2}}
		\frac{e^{\(y_{n+1} -\xi \)/2}}{\sqrt{t^2 -\lvert x -y' \rvert^2 -\lvert \xi - y_{n+1} \rvert^2 }}dy_{n+1} \)dy' \\
	& =\int_{B_t^n(x)} I_0 \( \frac{1}{2}  \sqrt{t^2 -\lvert x -y' \rvert^2} \) g\( y' \) dy' ,
\end{align*}
where $y= ( y' ,y_{n+1} ) \in \R^n \times \R$, we obtain the conclusion.

Finally, we consider even-dimensional cases. As we have
\begin{align*}
	&\frac{1}{t} \int_{S_t^n(z)} e^{-\xi /2} g(y) d \sigma_n (y) \\
	&=\frac{1}{t} \( \int_{S_t^n(z) \cap \left\{ y_{n+1} \geq \xi \right\}}
		+\int_{S_t^n(z) \cap \left\{ y_{n+1} \leq \xi \right\}} \) e^{-\xi /2} g(y) d\sigma_n (y) \\
	&=\int_{B_t^n(x)} \( e^{\sqrt{t^2 -\lvert x-y'\rvert^2}/2}
		+  e^{-\sqrt{t^2 -\lvert x-y'\rvert^2}/2} \) \frac{g\( y'\)}{\sqrt{t^2 -\lvert x-y'\rvert^2}} dy' \\
	&=2\int_{B_t^n(x)}
		\frac{\cosh \( \frac{1}{2} \sqrt{t^2 -\lvert x-y'\rvert^2}\)}{\sqrt{t^2 -\lvert x-y'\rvert^2}}
		g\( y' \) dy',
\end{align*}
the proof is completed.
\end{proof}

\begin{ex}\label{exsol}
{\rm We have the following formulas:
\begin{align*}
	S_1(t)g(x)&=\frac{e^{-t/2}}{2}\int_{x-t}^{x+t}
		I_0\(\frac{1}{2}\sqrt{t^2-r^2}\)g(y)dy,\\
	S_2(t)g(x)&=\frac{e^{-t/2}}{2\pi}
	\int_{B_t^2(x)}\frac{\cosh \( \frac{1}{2}\sqrt{t^2-r^2} \)}{\sqrt{t^2-r^2}}
	g(y)dy,\\
	S_3(t)g(x)&=\frac{e^{-t/2}}{4\pi}
	\( \frac{1}{t}\frac{\pd}{\pd t} \)
	\int_{B_t^3(x)}I_0\( \frac{1}{2}\sqrt{t^2-r^2}\)g(y)dy.
\end{align*}
These were given in [CH] in the same manner as in Proposition \ref{propsol}.}
\end{ex}

\begin{rem}\label{Duhamel}
{\rm Let $f$ and $g$ be smooth functions. Using the solution operator $S_n(t)$, we can express the solution of the Cauchy problem \eqref{DWfg} as
\[
u(x,t)=S_n(t) \( f+g \) (x)+\frac{\pd}{\pd t} S_n(t)f(x).
\]
}
\end{rem}
\subsection{Decomposition of the fundamental solution}

In this subsection, using Proposition \ref{propsol}, we give the Nishihara decomposed form of the solution operator $S_n(t)$.

\begin{thm}\label{thmdecom}
Let $g$ be a smooth function.
\begin{enumerate}[$(1)$]
\item Let $n$ be an odd number greater than one. Put
\begin{align*}
	k_\ell (s)&=\frac{1}{2^\ell}\sum_{j=0}^{\infty}
		\frac{1}{j!(j+\ell )!}\(\frac{s}{2}\)^{2j}
		=\frac{I_\ell (s)}{s^\ell},\\
	J_n(t)g(x)&=\frac{c_ne^{-t/2}}{2^{n-1}}
	\int_{B_t^n(x)}k_{\frac{n-1}{2}}\(\frac{1}{2}\sqrt{t^2-r^2}\)g(y)dy, \\
	\vect{W}_n(t)g(x)&=c_n
	\sum_{j=0}^{(n-3)/2}
	\frac{1}{8^j j!}\(\frac{1}{t}\frac{\pd}{\pd t}\)^{(n-3)/2-j}
	\(\frac{1}{t}\int_{S_t^{n-1}(x)}g(y)d\sigma_{n-1}(y) \) .
\end{align*}
Then, we have
\[
	S_n(t)g(x)=J_n(t)g(x)+e^{-t/2}\vect{W}_n(t)g(x) .
\]
Moreover, we have
\[
k_\ell (s) =\frac{1}{s^\ell} \frac{e^s}{\sqrt{2\pi s}} \(1-\frac{(\ell -1/2)(\ell +1/2)}{2s}+O\( \frac{1}{s^2} \) \) 
\]
as $s$ goes to infinity.
\item Let $n$ be an even number. Put
\begin{align*}
	k_\ell (s)
	&=\sum_{j=0}^{\infty}
		\frac{1}{ \( 2(j+\ell ) \) !!(2j+1)!!}s^{2j+1},\\
	J_n(t)g(x)
	&=
		\frac{c_ne^{-t/2}}{2^{n-2}}
		\int_{B_t^n(x)}k_{\frac{n}{2}}\(\frac{1}{2}\sqrt{t^2-r^2}\)g(y)dy, \\
	\vect{W}_n(t)g(x)
	&= 2 c_n
		\sum_{j=0}^{(n-2)/2}\frac{1}{8^j j!}
		\(\frac{1}{t}\frac{\pd}{\pd t}\)^{(n-2)/2-j}
		\int_{B_t^n(x)}\frac{1}{\sqrt{t^2-r^2}}g(y)dy.
\end{align*}
Then, we have
\[
	S_n(t)g(x)=J_n(t)g(x)+e^{-t/2}\vect{W}_n(t)g(x) .
\]
Moreover, we have
\[
k_\ell (s) =\frac{e^s}{2s^\ell}\(1-\frac{\ell (\ell -1)}{2s}+O\( \frac{1}{s^2} \)\) 
\]
as $s$ goes to infinity.
\end{enumerate}
\end{thm}

\begin{rem}
\label{recursion}
{\rm 
\begin{enumerate}[(1)]
\item Let $n$ be an odd number greater than one. From the fact \eqref{Bessel_2}, the kernel $k_\ell (s)$ has the following properties:
\[
k_{\ell +1} (s) = \frac{k_\ell'(s)}{s} ,\ k_\ell (0) = \frac{1}{2^\ell \ell !},\ k_\ell' (0) =0 .
\]
\item Let $n$ be an even number. In the proof of Theorem \ref{thmdecom}, we will show that the kernel $k_\ell (s)$ is defined by the following recursion:
\[
k_1(s)=\frac{\cosh(s)-1}{s},\ k_\ell (s)=\frac{k_{\ell-1}' (s)-k_{\ell-1}' (0)}{s}.
\]
In particular, we have the following properties:
\[
k_\ell (0)=0,\ k_\ell' (0)=\frac{1}{2^\ell \ell !}.
\]
\end{enumerate}
}
\end{rem}

\begin{proof}[Proof of Theorem \ref{thmdecom}]
(1) The solution formula in Proposition \ref{propsol} implies
\begin{align*}
	S_n(t)g(x)
	&=c_n e^{-t/2}
	\(\frac{1}{t}\frac{\pd}{\pd t}\)^{(n-3)/2}
	\(\frac{1}{t}\int_{S_t^{n-1}(x)}g(y)d\sigma_{n-1}(y) \)\\
	&\quad + c_n e^{-t/2} 
	\(\frac{1}{t}\frac{\pd}{\pd t}\)^{(n-3)/2}
	\int_{B_t^n(x)}\frac{I_1\(\frac{1}{2}\sqrt{t^2-r^2}\)}{2\sqrt{t^2-r^2}}g(y)dy.
\end{align*}
Note that the second term of the right-hand side can be written as
\[
	\frac{c_n e^{-t/2}}{4}
	\(\frac{1}{t}\frac{\pd}{\pd t}\)^{(n-3)/2}
	\int_{B_t^n(x)}k_1\(\frac{1}{2}\sqrt{t^2-r^2}\)g(y)dy.
\]
Since we have
\[
	\frac{1}{t}\(\frac{\pd}{\pd t}k_1\(\frac{1}{2}\sqrt{t^2-r^2}\)\)
	=\frac{k_1' \(\frac{1}{2}\sqrt{t^2-r^2}\)}{2\sqrt{t^2-r^2}}
	=\frac{1}{4}k_2\(\frac{1}{2}\sqrt{t^2-r^2}\right),
\]
we obtain
\begin{align*}
	&\frac{c_ne^{-t/2}}{4}
	\(\frac{1}{t}\frac{\pd}{\pd t}\)^{(n-3)/2}
	\int_{B_t^n(x)}k_1\(\frac{1}{2}\sqrt{t^2-r^2}\)g(y)dy\\
	&=\frac{c_ne^{-t/2}}{4}
	\(\frac{1}{t}\frac{\pd}{\pd t}\)^{(n-5)/2}
	\(\frac{k_1(0)}{t}\int_{S_t^{n-1}(x)}g(y)d\sigma_{n-1}(y) \)\\
	&\quad+
	\frac{c_n e^{-t/2}}{4^2}
	\(\frac{1}{t}\frac{\pd}{\pd t}\)^{(n-5)/2}
	\int_{B_t^n(x)}k_2\(\frac{1}{2}\sqrt{t^2-r^2}\)g(y)dy.
\end{align*}
Continuing this argument, we obtain the decomposed form of $S_n(t)g$.

The asymptotic expansions are direct consequences from \eqref{Bessel_3} and $k_{\ell+1}(s)=k_\ell' (s)/s$.

(2) By the solution formula in Proposition \ref{propsol}, we have
\begin{align*}
	S_n(t)g(x)
	&=2c_n e^{-t/2}
		\(\frac{1}{t}\frac{\pd}{\pd t}\)^{(n-2)/2}
		\int_{B_t^n(x)}\frac{1}{\sqrt{t^2-r^2}}g(y)dy\\
	&\quad + 2c_n e^{-t/2}
		\(\frac{1}{t}\frac{\pd}{\pd t}\)^{(n-2)/2}
		\int_{B_t^n(x)} \frac{\cosh \(\frac{1}{2}\sqrt{t^2-r^2}\)-1}{\sqrt{t^2-r^2}} g(y)dy.\\
	&=: V_1+K_1.
\end{align*}
Here, we note that
$K_1$
can be written as
\begin{align*}
	K_1
	&= c_n e^{-t/2}
		\(\frac{1}{t}\frac{\pd}{\pd t}\)^{(n-2)/2}
		\int_{B_t^n(x)}k_1\(\frac{1}{2}\sqrt{t^2-r^2}\) g(y)dy\\
	&= c_n e^{-t/2}
		\(\frac{1}{t}\frac{\pd}{\pd t}\)^{(n-4)/2}
		\int_{B_t^n(x)}
		\frac{k_1' \(\frac{1}{2}\sqrt{t^2-r^2}\)}{2\sqrt{t^2-r^2}} g(y)dy\\
	&= \frac{c_n e^{-t/2}}{2} k_1' (0)
		\(\frac{1}{t}\frac{\pd}{\pd t}\)^{(n-4)/2}
		\int_{B_t^n(x)}
		\frac{1}{\sqrt{t^2-r^2}} g(y)dy\\
	&\quad + \frac{c_n e^{-t/2}}{4}
		\(\frac{1}{t}\frac{\pd}{\pd t}\)^{(n-4)/2}
		\int_{B_t^n(x)}
		k_2\(\frac{1}{2}\sqrt{t^2-r^2}\) g(y)dy\\
	&=: V_2+K_2.
\end{align*}
In the same manner as in the calculation of $K_1$, we have
\begin{align*}
	K_2
	&= \frac{c_n e^{-t/2}}{2^3} k_2' (0)
		\(\frac{1}{t}\frac{\pd}{\pd t}\)^{(n-6)/2}
		\int_{B_t^n(x)}
		\frac{1}{\sqrt{t^2-r^2}} g(y)dy\\
	&\quad +\frac{c_n e^{-t/2}}{4^2}
		\(\frac{1}{t}\frac{\pd}{\pd t}\)^{(n-6)/2}
		\int_{B_t^n(x)}
		k_3\(\frac{1}{2}\sqrt{t^2-r^2}\) g(y)dy\\
	&=: V_3+K_3.
\end{align*}
Continuing this argument, we obtain
\[
	V_j = \frac{c_n e^{-t/2}}{2^{2 j -3}} k_{j-1}' (0)
	\(\frac{1}{t}\frac{\pd}{\pd t}\)^{n/2-j}
	\int_{B_t^n(x)}
	\frac{1}{\sqrt{t^2-r^2}} g(y)dy
\]
and the decomposed form of $S_n(t) g$.

The asymptotic expansions follow from
\[
	k_1(s)=\frac{e^s}{2s}+\frac{e^{-s}}{2s}-\frac{1}{s}
\]
and the recursion of $k_\ell (s)$.
\end{proof}

\begin{rem}\label{decom1}
{\rm 
Let $g$ be a smooth function. Put
\begin{align*}
J_1(t)g(x)&=\frac{e^{-t/2}}{2}\int_{x-t}^{x+t}
		\( I_0\(\frac{1}{2}\sqrt{t^2-r^2}\)-1 \) g(y)dy,\\
\vect{W}_1(t)g(x) &=W_1(t)g(x) =\frac{1}{2}\int_{x-t}^{x+t}g(y)dy .
\end{align*}
Thanks to Proposition \ref{propsol}, we get the Nishihara decomposed form of $S_1(t)$ as
\[
S_1(t)g(x) = J_1(t)g (x) +e^{-t/2} \vect{W}_1(t)g(x)
\]
which was given in [MN].
}
\end{rem}

\begin{ex}\label{decom23}
{\rm We have the following formulas:
\begin{align*}
	J_2(t)g(x)&=\frac{e^{-t/2}}{2\pi}\int_{B_t^2(x)}
		\frac{\cosh \( \frac{1}{2}\sqrt{t^2-r^2} \)-1}{\sqrt{t^2-r^2}}
		g(y)dy,\\
	J_3(t)g(x)&=\frac{e^{-t/2}}{4\pi} \int_{B_t^3(x)}
		\frac{I_1\( \frac{1}{2}\sqrt{t^2-r^2} \)}{2\sqrt{t^2-r^2}}
		g(y)dy,\\
	\vect{W}_2(t)g(x) &=W_2(t)g(x) =\frac{1}{2\pi}\int_{B_t^2(x)}
		\frac{1}{\sqrt{t^2-r^2}}g(y)dy,\\
	\vect{W}_3(t)g(x) &=W_3(t)g(x) =\frac{1}{4\pi t}\int_{S_t^2 (x)}g(y)d\sigma_2 (y).
\end{align*}
These were given in [HO, Nis] using the explicit form of $S_n(t)g(x)$ given in Example \ref{exsol}.
}
\end{ex}

In [HO, MN, Nis], 
it was shown that
$J_n(t)g$
behaves like
$P_n(t)g$
as
$t$ goes to infinity.
More precisely, 
in the case of $1\leq n \leq 3$, 
for $t>0$ and $1\le q\le p\le\infty$, 
the $L^p$-$L^q$ estimate
\begin{equation}\label{estimate_low}
	\|J_n(t)g-P_n(t)g\|_{L^p}\leq
		Ct^{-\frac{n}{2}\(\frac{1}{q}-\frac{1}{p}\)-1} \| g \|_{L^q}
\end{equation}
was shown.

The Nishihara decomposition and the estimate \eqref{estimate_low}
imply the following properties of the damped wave equation
(see [Nis, pp. 632]):
\begin{itemize}
	\item The damped wave equation does not have the smoothing effect,
and the singularity of the initial datum propagates along the light cone
by the wave property.
However, the strength of the singularity decays exponentially
by the damping effect.
	\item If the initial datum is sufficiently smooth, then
the damped wave equation may have the same properties as those to parabolic
equations under some suitable situations. 
\end{itemize}

For the case of $n\ge 4$, in [Nar], Narazaki proved a similar decomposition to Theorem \ref{thmdecom} in the Fourier space.
However, the explicit form of the decomposition in the configuration space was not known. In section 5, using Theorem \ref{thmdecom}, we will give an estimate of the difference $S_n(t)g- P_n(t)g -e^{-t/2}\vect{W}_n(t)g$, which is slightly sharper than that of [Nar].

\subsection{Decomposition of the solution with general initial datum}
In this subsection, we give the Nishihara decomposed form of the solution of \eqref{DWfg}.

\begin{prop}\label{prop_de_fg}
Let $f$ and $g$ be smooth functions, $h=f+g$, and $u$ the unique classical solution of the Cauchy problem \eqref{DWfg}.
\begin{enumerate}[$(1)$]
\item Let $n=1$. Put
\begin{align*}
	\tilde{J}_1(t)f(x)
	&=\frac{e^{-t/2}}{4}\int_{x-t}^{x+t} \( \frac{t}{\sqrt{t^2-r^2}}
			I_1 \( \frac{1}{2}\sqrt{t^2-r^2} \)  -I_0 \( \frac{1}{2}\sqrt{t^2-r^2} \) \) f(y) dy, \\
	\widehat{\vect{W}}_1(t)f(x) &= \frac{1}{2} \( f(x+t) + f(x-t) \) ,\\
	\widetilde{\vect{W}}_1(t;f,g) (x) &= \vect{W}_1(t) h(x)  + \widehat{\vect{W}}_1(t)f(x) .
\end{align*}
Then, we have
\[
\frac{\pd}{\pd t} S_1 (t) f(x) = \tilde{J}_1 (t) f(x) + e^{-t/2} \widehat{\vect{W}}_1(t) f(x).
\]
In other words, the solution $u(x,t)$ is expressed as 
\[
	u(x,t) = J_1(t) h(x) + \tilde{J}_1(t) f( x ) + e^{-t/2}\widetilde{\vect{W}}_1(t;f,g)(x).
\]
\item Let $n$ be an odd number greater than one. Put
\begin{align*}
	\tilde{J}_n(t) f(x) &=
	\frac{c_n e^{-t/2}}{2^{n+1}}
		\int_{B_t^n(x)}
				\( t k_{\frac{n+1}{2}}\(\frac{1}{2}\sqrt{t^2-r^2}\) -2k_{\frac{n-1}{2}} \( \frac{1}{2}\sqrt{t^2 -r^2} \) \) f(y)dy ,\\
	\widehat{\vect{W}}_n(t)f(x) &=
	\frac{c_n}{2^{\frac{3(n-1)}{2}} \( \frac{n-1}{2} \) !}
			\int_{S_t^{n-1}(x)}f(y)d\sigma_{n-1}(y) ,\\
	\widetilde{\vect{W}}_n(t; f, g)(x) &= \frac{1}{2}\vect{W}_n(t)f(x) + \vect{W}_n(t)g(x)
			+ \widehat{\vect{W}}_n(t)f(x)  + \frac{\pd}{\pd t} \vect{W}_n(t)f(x) .
\end{align*}
Then, we have
\[
	\frac{\pd}{\pd t} S_n(t) f(x)
	=\tilde{J}_n(t)f(x) + e^{-t/2} \widehat{\vect{W}}_n(t)f(x)
		-\frac{e^{-t/2}}{2} \vect{W}_n(t)f(x) + e^{-t/2} \frac{\pd}{\pd t} \vect{W}_n(t)f(x).
\]
In other words, the solution $u(x,t)$ is expressed as
\[
	u( x, t )=J_n(t) h (x) + \tilde{J}_n(t) f (x) + e^{-t/2} \widetilde{\vect{W}}_n(t; f, g)(x).
\]
\item Let $n$ be an even number. Put
\begin{align*}
	\tilde{J}_n(t) f(x) &=
	\frac{c_n e^{-t/2}}{2^n}
		\int_{B_t^n(x)}\( t k_{\frac{n}{2}+1}\(\frac{1}{2}\sqrt{t^2-r^2}\)
		-2k_{\frac{n}{2}}\(\frac{1}{2}\sqrt{t^2-r^2}\)\)f(y)dy, \\
	\widehat{\vect{W}}_n(t)f(x) &=
	\frac{c_n t}{2^{\frac{3n-2}{2}} \( \frac{n}{2} \) !}
			\int_{B_t^n(x)}\frac{1}{\sqrt{t^2-r^2}} f(y) dy, \\
	\widetilde{\vect{W}}_n(t; f, g)(x) &= \frac{1}{2}\vect{W}_n(t)f(x) + \vect{W}_n(t)g(x)
			+ \widehat{\vect{W}}_n(t)f(x) + \frac{\pd}{\pd t} \vect{W}_n(t)f(x) .
\end{align*}
Then, we have
\[
	\frac{\pd}{\pd t} S_n(t) f(x)
	=\tilde{J}_n(t)f(x) + e^{-t/2} \widehat{\vect{W}}_n(t)f(x)
		-\frac{e^{-t/2}}{2} \vect{W}_n(t)f(x) + e^{-t/2} \frac{\pd}{\pd t} \vect{W}_n(t)f(x).
\]
In other words, the solution $u(x,t)$ is expressed as
\[
	u( x, t )=J_n(t) h (x) + \tilde{J}_n(t) f (x) + e^{-t/2} \widetilde{\vect{W}}_n(t; f, g)(x).
\]
\end{enumerate}
\end{prop}

\begin{proof}
(1) From Proposition \ref{propsol}, direct computation shows
\begin{align*}
	\frac{\pd}{\pd t}S_1(t) f(x) &=
		\frac{e^{-t/2}}{4}\int_{x-t}^{x+t} \( \frac{t}{\sqrt{t^2-r^2}} I'_0 \( \frac{1}{2}\sqrt{t^2-r^2} \)-I_0 \( \frac{1}{2}\sqrt{t^2-r^2} \) \) f(y) dy\\
		&\quad +\frac{e^{-t/2}}{2} \( f(x+t) + f(x-t) \) \\
	&= \tilde{J}_1(t)f(x) + e^{-t/2}\widehat{\vect{W}}_1(t)f(x)  .
\end{align*}
Here, we used $I_0'(s) = I_1(s)$ (see \eqref{Bessel_2}). From Remark \ref{Duhamel}, we get the conclusion.

(2) From Remark \ref{recursion}, we have
\begin{align*}
	\frac{\pd}{\pd t}J_n(t)f(x)&=\frac{c_n}{2^{n-1}}
		\int_{B_t^n(x)}\frac{\pd}{\pd t}
				\( e^{-t/2}k_{\frac{n-1}{2}}\(\frac{1}{2}\sqrt{t^2-r^2}\)\) f(y)dy\\
		&\quad +\frac{c_n e^{-t/2}}{2^{\frac{3(n-1)}{2}} \( \frac{n-1}{2} \) !}
			\int_{S_t^{n-1}(x)}f(y)d\sigma_{n-1}(y)\\
		&= \tilde{J}_n(t)g(x) + e^{-t/2}\widehat{\vect{W}}_n(t)g(x).
\end{align*}
From Remark \ref{Duhamel}, we get the conclusion.

(3) From Remark \ref{recursion}, we have
\begin{align*}
\frac{\pd}{\pd t} J_n(t)g(x)
&=\frac{c_n}{2^{n-2}} \int_{B_t^n(x)}\frac{\pd}{\pd t} \(e^{-t/2} k_{\frac{n}{2}}\(\frac{1}{2}\sqrt{t^2-r^2}\)\)g(y)dy \\
&= \tilde{J}_n(t)g(x) +e^{-t/2}\widehat{\vect{W}}_n(t)g(x).
\end{align*}
From Remark \ref{Duhamel}, we get the conclusion.
\end{proof}

\section{Preliminary estimates}
In this section, we prepare some estimates
which will be frequently used for studying the behavior of time-delayed hot spots.
We postpone the proofs of the following lemmas until Appendix 1. This is because their proofs consist of tedious calculations.

\begin{lem}\label{lem_est_Wn}
Let $f$ and $g$ be smooth bounded functions defined on $\R^n$.

\begin{enumerate}[$(1)$]
\item For any $x \in \R$, $t >0$ and $\al = (\al_1 , \al_2 ) \in \Z_{\geq 0}^2$, we have the following inequalities:
\begin{align*}
	\lvert \frac{\pd^{\lvert \al \rvert}}{\pd x^{\al_1} \pd t^{\al_2}}\vect{W}_1(t)g(x) \rvert
		& \le (1+t) \| g\|_{W^{\lvert \al \rvert, \infty}} ,\\
	\lvert \frac{\pd^{\lvert \al \rvert}}{\pd x^{\al_1} \pd t^{\al_2}} \widehat{\vect{W}}_1(t)f(x) \rvert
	&\le \| f \|_{W^{\lvert \al \rvert,\infty}},\\
      \lvert \frac{\pd^{\lvert \al \rvert}}{\pd x^{\al_1} \pd t^{\al_2}} \widetilde{\vect{W}}_1(t;f,g)(x) \rvert
	&\le 2 (1+t) \( \| f \|_{W^{\lvert \al \rvert , \infty}} + \| g \|_{W^{\lvert \al \rvert,\infty}} \) ,
\end{align*}
where $\lvert \al \rvert=\al_1 +\al_2$.
\item Let $n$ be an odd number greater than one. There exists a positive constant $C=C(n)$ such that, for any $x \in \R^n$, $t  >0$ and $\al = (\al_1 ,\ldots , \al_{n+1}) \in \Z_{\geq 0}^{n+1}$, we have the following inequalities:
\begin{align*}
	\lvert \frac{\pd^{\lvert \al \rvert}}{\pd x_1^{\al_1} \cdots \pd x_n^{\al_n} \pd t^{\al_{n+1}}}\vect{W}_n(t)g(x) \rvert
		&\le C (1+t)^{n-2} \| g\|_{W^{(n-3)/2+\lvert \al \rvert, \infty}} ,\\
	\lvert \frac{\pd^{\lvert \al \rvert}}{\pd x_1^{\al_1} \cdots \pd x_n^{\al_n} \pd t^{\al_{n+1}}} \widehat{\vect{W}}_n(t)f(x) \rvert
	&\le C (1+t)^{n-1} \| f \|_{W^{\lvert \al \rvert,\infty}},\\
      \lvert \frac{\pd^{\lvert \al \rvert}}{\pd x_1^{\al_1} \cdots \pd x_n^{\al_n} \pd t^{\al_{n+1}}} \widetilde{\vect{W}}_n(t;f,g)(x) \rvert
	&\le C (1+t)^{n-1} \(  \| f \|_{W^{(n-1)/2+\lvert \al \rvert,\infty}} +  \| g \|_{W^{(n-3)/2+\lvert \al \rvert,\infty}} \) ,
\end{align*}
where $\lvert \al \rvert =\al_1 + \cdots + \al_{n+1}$.
\item Let $n$ be an even number greater than one. There exists a positive constant $C=C(n)$ such that, for any $x \in \R^n$, $t >0$ and $\al = (\al_1 ,\ldots , \al_{n+1}) \in \Z_{\geq 0}^{n+1}$, we have the following inequalities:
\begin{align*}
	\lvert \frac{\pd^{\lvert \al \rvert}}{\pd x_1^{\al_1} \cdots \pd x_n^{\al_n} \pd t^{\al_{n+1}}} \vect{W}_n(t)g(x) \rvert
		&\le C(1+t)^{n-1} \| g\|_{W^{n/2-1+\lvert \al \rvert, \infty}} ,\\
	\lvert \frac{\pd^{\lvert \al \rvert}}{\pd x_1^{\al_1} \cdots \pd x_n^{\al_n} \pd t^{\al_{n+1}}} \widehat{\vect{W}}_n(t)f(x) \rvert
	&\le C (1+t)^n \| f \|_{W^{\lvert \al \rvert,\infty}},\\
      \lvert \frac{\pd^{\lvert \al \rvert}}{\pd x_1^{\al_1} \cdots \pd x_n^{\al_n} \pd t^{\al_{n+1}}} \widetilde{\vect{W}}_n(t;f,g)(x) \rvert
	&\le C (1+t)^n \( \| f \|_{W^{n/2 +\lvert  \al \rvert,\infty}}  + \| g \|_{W^{n/2 -1+\lvert  \al \rvert,\infty}} \) ,
\end{align*}
where $\lvert \al \rvert =\al_1 + \cdots + \al_{n+1}$.
\end{enumerate}
\end{lem}

\begin{lem}\label{list_J}
Let $h:\R^n \to \R$ be a smooth function.
\begin{enumerate}[$(1)$]
\item Let $n=1$. We have the following identities:
\begin{align*}
\frac{\pd}{\pd x} J_1(t)h(x)
&=-\frac{e^{-t/2}}{4} \int_{x-t}^{x+t} \frac{1}{\sqrt{t^2-r^2}} I_1 \( \frac{1}{2} \sqrt{t^2-r^2} \) h(y) (x-y) dy ,\\
\frac{\pd^2}{\pd x^2} J_1(t)h(x) 
&=\frac{te^{-t/2}}{8} \( h(x+t) +h(x-t) \) \\
&\quad -\frac{e^{-t/2}}{4} \int_{x-t}^{x+t} \frac{1}{\sqrt{t^2-r^2}} I_1 \( \frac{1}{2} \sqrt{t^2-r^2} \) h(y) dy \\
&\quad +\frac{e^{-t/2}}{8} \int_{x-t}^{x+t} \frac{1}{t^2-r^2} I_2 \( \frac{1}{2} \sqrt{t^2-r^2} \) h(y) (x-y)^2 dy .
\end{align*}
\item Let $n$ be an odd number greater than one. For any direction $\ome \in S^{n-1}$, we have the following identities:
\begin{align*}
\nabla J_n(t)h(x)
&= \frac{c_ne^{-t/2}}{2^{n-1}} k_{\frac{n-1}{2}}(0) t^{n-1} \int_{S^{n-1}} h(x+t\theta ) \theta d\sigma_{n-1} (\theta ) \\
&\quad - \frac{c_ne^{-t/2}}{2^{n+1}} \int_{B^n_t(x)} k_{\frac{n+1}{2}} \( \frac{1}{2} \sqrt{t^2 -r^2} \) h(y) (x-y) dy ,\\
(\ome \cdot \nabla )^2 J_n(t) h(x) 
&= \frac{c_ne^{-t/2}}{2^{n-1}} k_{\frac{n-1}{2}}(0) t^{n-1} \int_{S^{n-1}} \ome \cdot \nabla h(x+t\theta ) \ome \cdot \theta d\sigma_{n-1} (\theta ) \\
&\quad + \frac{c_ne^{-t/2}}{2^{n+1}} k_{\frac{n+1}{2}}(0) t^n \int_{S^{n-1}} h(x+t\theta ) \( \ome \cdot \theta \)^2 d\sigma_{n-1} (\theta ) \\
&\quad - \frac{c_ne^{-t/2}}{2^{n+1}} \int_{B^n_t(x)} k_{\frac{n+1}{2}} \( \frac{1}{2} \sqrt{t^2 -r^2} \) h(y)  dy \\
&\quad + \frac{c_ne^{-t/2}}{2^{n+3}} \int_{B^n_t(x)} k_{\frac{n+3}{2}} \( \frac{1}{2} \sqrt{t^2 -r^2} \) h(y) \( \ome \cdot (x-y) \)^2 dy.
\end{align*}
\item Let $n$ be an even number. For any direction $\ome \in S^{n-1}$, we have the following identities:
\begin{align*}
\nabla J_n(t)h(x)
&= \frac{c_ne^{-t/2}}{2^{n-1}} k'_{\frac{n}{2}}(0) t^n \int_{B^n} \frac{1}{\sqrt{1-\lvert z \rvert^2}} h(x+tz ) z dz \\
&\quad - \frac{c_ne^{-t/2}}{2^n} \int_{B^n_t(x)} k_{\frac{n}{2}+1} \( \frac{1}{2} \sqrt{t^2 -r^2} \) h(y) (x-y) dy ,\\
(\ome \cdot \nabla )^2 J_n(t) h(x) 
&= \frac{c_n e^{-t/2}}{2^{n-1}} k'_{\frac{n}{2}}(0) t^n \int_{B^n} \frac{1}{\sqrt{1-\lvert z \rvert^2}} \ome \cdot \nabla h(x+tz) \ome \cdot z dz \\
&\quad + \frac{c_n e^{-t/2}}{2^{n+1}} k'_{\frac{n}{2}+1}(0) t^{n+1} \int_{B^n} \frac{1}{\sqrt{1-\lvert z \rvert^2}} h(x+tz ) \( \ome \cdot z \)^2 dz \\
&\quad - \frac{c_ne^{-t/2}}{2^n} \int_{B^n_t(x)} k_{\frac{n}{2}+1} \( \frac{1}{2} \sqrt{t^2 -r^2} \) h(y)  dy \\
&\quad + \frac{c_ne^{-t/2}}{2^{n+2}} \int_{B^n_t(x)} k_{\frac{n}{2}+2} \( \frac{1}{2} \sqrt{t^2 -r^2} \) h(y) \( \ome \cdot (x-y) \)^2 dy.
\end{align*}
\end{enumerate}
\end{lem}

\begin{lem}\label{En} 
Let $h$ be a smooth function with compact support. Put
\begin{align*}
	E_n(r,t)=
	\begin{cases}
	\ds \frac{ e^{-t/2}}{4}\frac{1}{\sqrt{t^2-r^2}} I_1 \( \frac{1}{2}\sqrt{t^2-r^2} \)
	&(n=1),\\
	\ds \frac{c_n e^{-t/2}}{2^{n+1}}
		k_{\frac{n+1}{2}} \(\frac{1}{2}\sqrt{t^2-r^2}\)
	&(n\in 2\N +1),\\
	\ds \frac{c_n e^{-t/2}}{2^n}
		k_{\frac{n}{2}+1}\( \frac{1}{2}\sqrt{t^2-r^2} \)
	&(n\in 2\N ).
	\end{cases}
\end{align*}
If $x\in CS(h) + (t-d_h)B^n$ and $t \geq d_h$, then we have
\[
	\nabla J_n(t)h(x)=-\int_{B_t^n(x)}E_n(r,t)h(y) (x-y) dy .
\]
\end{lem}

\begin{lem}\label{En_asymp}
Let $\f :[0, +\infty ) \to [0,+\infty )$ be a non-decreasing function with $\f (0)=0$.
\begin{enumerate}[$(1)$]
\item If $\sqrt{t^2-\varphi(t)^2}$ diverges as $t$ goes to infinity, then we have
\[
	E_n \( \f (t),t \)=\frac{1}{2(4\pi)^{n/2} \( t^2-\f (t)^2 \)^{n/4+1/2}}
		\exp \( \frac{-t+\sqrt{t^2-\f (t)^2}}{2} \) 
		\( 1+O\( \frac{1}{\sqrt{t^2-\f (t)^2}} \) \)
\]
as $t$ goes to infinity.
\item If $\f(t)$ is of small order of $t$ as $t$ goes to infinity, then we have
\[
	E_n \( \f (t),t \) =\frac{1}{2(4\pi)^{n/2}t^{n/2+1}}
	\exp \( \frac{-t+\sqrt{t^2-\f (t)^2}}{2} \)  
		 \( 1+O\( \frac{1}{t} \) + O\( \frac{\varphi(t)^2}{t^2} \) \)
\]
as $t$ goes to infinity.
\item If $\f (t)$ is of small order of $\sqrt{t}$, then we have
\[
	E_n \( \varphi(t),t \) = \frac{1}{2(4\pi)^{n/2}t^{n/2+1}} \( 1+ O \( \frac{1}{t} \) + O \( \frac{\f (t)^2}{t} \) \) 	
\]
as $t$ goes to infinity.
\end{enumerate}
\end{lem}

\begin{lem}\label{list_tildeJ}
Let $f:\R^n \to \R$ be a smooth function.
\begin{enumerate}[$(1)$]
\item Let $n=1$. We have the following identities:
\begin{align*}
&\frac{\pd}{\pd x} \tilde{J}_1(t) f(x) \\
&= \frac{(t-4)e^{-t/2}}{16}  \( f(x+t) -f(x-t) \) \\ 
&\quad - \frac{e^{-t/2}}{8} \int_{x-t}^{x+t} \frac{1}{\sqrt{t^2 -r^2}} \( \frac{t}{\sqrt{t^2-r^2}} I_2 \( \frac{1}{2} \sqrt{t^2-r^2} \) -   I_1 \( \frac{1}{2} \sqrt{t^2-r^2} \) \) f(y) (x-y) dy,\\
&\frac{\pd^2}{\pd x^2} \tilde{J}_1(t) f(x) \\
&= \frac{(t-4)e^{-t/2}}{16}  \( f' (x+t) -f' (x-t) \) \\
&\quad + \frac{\( t^2-8t \)e^{-t/2}}{256}  \( f(x+t)+f(x-t) \) \\
&\quad - \frac{e^{-t/2}}{8} \int_{x-t}^{x+t} \frac{1}{\sqrt{t^2 -r^2}} \( \frac{t}{\sqrt{t^2-r^2}} I_2 \( \frac{1}{2} \sqrt{t^2-r^2} \) -   I_1 \( \frac{1}{2} \sqrt{t^2-r^2} \) \) f(y)  dy \\
&\quad +\frac{e^{-t/2}}{16} \int_{x-t}^{x+t} \frac{1}{\( t^2-r^2\)} \( \frac{t}{\sqrt{t^2-r^2}} I_3 \( \frac{1}{2} \sqrt{t^2-r^2} \) -   I_2 \( \frac{1}{2} \sqrt{t^2-r^2} \) \) f(y) \( x-y \)^2 dy.
\end{align*}
\item Let $n$ be an odd number greater than one. For any direction $\ome \in S^{n-1}$, we have the following identities:
\begin{align*}
&\nabla \tilde{J}_n(t)f(x)\\
&=\frac{c_ne^{-t/2}}{2^{n+1}}\( tk_{\frac{n+1}{2}}(0) -2k_{\frac{n-1}{2}}(0) \) t^{n-1} \int_{S^{n-1}} f(x+t\theta ) \theta d\sigma_{n-1} (\theta ) \\
&\quad - \frac{c_n e^{-t/2}}{2^{n+3}} \int_{B^n_t (x)} \( t k_{\frac{n+3}{2}} \( \frac{1}{2} \sqrt{t^2-r^2} \) -2 k_{\frac{n+1}{2}} \( \frac{1}{2} \sqrt{t^2-r^2} \) \) f(y) (x-y) dy,\\
&\( \ome \cdot \nabla \)^2 \tilde{J}_n(t)f(x)\\
&=\frac{c_ne^{-t/2}}{2^{n+1}}\( tk_{\frac{n+1}{2}}(0) -2k_{\frac{n-1}{2}}(0) \) t^{n-1} \int_{S^{n-1}} \ome \cdot \nabla f(x+t\theta ) \ome \cdot \theta  d\sigma_{n-1} (\theta ) \\
&\quad + \frac{c_ne^{-t/2}}{2^{n+3}}\( tk_{\frac{n+3}{2}}(0) -2k_{\frac{n+1}{2}}(0) \) t^n \int_{S^{n-1}} f(x+t\theta )  (\ome \cdot \theta )^2 d\sigma_{n-1} (\theta ) \\
&\quad - \frac{c_n e^{-t/2}}{2^{n+3}} \int_{B^n_t (x)} \( t k_{\frac{n+3}{2}} \( \frac{1}{2} \sqrt{t^2-r^2} \) -2 k_{\frac{n+1}{2}} \( \frac{1}{2} \sqrt{t^2-r^2} \) \) f(y)  dy \\
&\quad + \frac{c_n e^{-t/2}}{2^{n+5}} \int_{B^n_t (x)} \( t k_{\frac{n+5}{2}} \( \frac{1}{2} \sqrt{t^2-r^2} \) -2 k_{\frac{n+3}{2}} \( \frac{1}{2} \sqrt{t^2-r^2} \) \) f(y) \( \ome \cdot (x-y ) \)^2 dy .
\end{align*}
\item Let $n$ be an even number. For any direction $\ome \in S^{n-1}$, we have the following identities:
\begin{align*}
&\nabla \tilde{J}_n(t)f(x)\\
&=\frac{c_ne^{-t/2}}{2^{n+1}}\( tk'_{\frac{n}{2}+1}(0) -2k'_{\frac{n}{2}}(0) \) t^n \int_{B^n} \frac{1}{\sqrt{1-\lvert z \rvert^2}} f(x+tz) z dz \\
&\quad - \frac{c_n e^{-t/2}}{2^{n+2}} \int_{B^n_t (x)} \( t k_{\frac{n}{2}+2} \( \frac{1}{2} \sqrt{t^2-r^2} \) -2 k_{\frac{n}{2}+1} \( \frac{1}{2} \sqrt{t^2-r^2} \) \) f(y) (x-y) dy,\\
&\( \ome \cdot \nabla \)^2 \tilde{J}_n(t)f(x)\\
&=\frac{c_ne^{-t/2}}{2^{n+1}}\( tk'_{\frac{n}{2}+1}(0) -2k'_{\frac{n}{2}}(0) \) t^n \int_{B^n} \frac{1}{\sqrt{1-\lvert z \rvert^2}} \ome \cdot \nabla f(x+tz ) \ome \cdot z  dz \\
&\quad + \frac{c_ne^{-t/2}}{2^{n+3}}\( tk'_{\frac{n}{2}+3}(0) -2k'_{\frac{n}{2}+1}(0) \) t^{n+1} \int_{B^n} \frac{1}{\sqrt{1-\lvert z \rvert^2}} f(x+tz)  \( \ome \cdot z \)^2 dz \\
&\quad - \frac{c_n e^{-t/2}}{2^{n+2}} \int_{B^n_t (x)} \( t k_{\frac{n}{2}+2} \( \frac{1}{2} \sqrt{t^2-r^2} \) -2 k_{\frac{n}{2}+1} \( \frac{1}{2} \sqrt{t^2-r^2} \) \) f(y)  dy \\
&\quad + \frac{c_n e^{-t/2}}{2^{n+4}} \int_{B^n_t (x)} \( t k_{\frac{n}{2}+3} \( \frac{1}{2} \sqrt{t^2-r^2} \) -2 k_{\frac{n}{2}+2} \( \frac{1}{2} \sqrt{t^2-r^2} \) \) f(y) \( \ome \cdot (x-y ) \)^2 dy .
\end{align*}
\end{enumerate}
\end{lem}

\begin{lem}\label{lem_tildeJ}
Let $f$ be a non-zero smooth function with compact support. 
\begin{enumerate}[$(1)$]
\item There exists a positive constant $C=C(n)$ such that, for any $x \in \R^n$ and $t > 0$, we have
\[
	\lvert \tilde{J}_n(t) f(x) \rvert \le C(1+ t)^{-n/2-1} \| f \|_{L^1} .
\]
\item Let $\psi :[0, +\infty ) \to [0,+\infty )$ be a non-decreasing function with $\psi (0)=0$. Suppose that $\psi (t)$ is of small order of $\sqrt{t}$. Let 
\[
T_0 (\psi ) = \min \left\{ T>0 \lvert \forall t\geq T,\ t \geq \psi (t) +d_f \right\} \right. .
\]
There exists a positive constant $C=C(n,d_f ,\psi )$ such that, for any $x \in CS(f) +\psi (t) B^n$ and $t\geq T_0 (\psi )$, we have
\[
\lvert \nabla \tilde{J}_n(t) f (x) \rvert \leq C \( 1+t \)^{-n/2 -3} \( 1+ t+ \psi (t)^2 \) \( 1+ \psi (t)  \) \| f \|_{L^1} .
\]
\item Let $R$ be a positive constant. There exists a positive constant $C=C(n,d_f,R)$ such that, for any $x \in CS(f)+RB^n$, $\ome \in S^{n-1}$ and $t \geq R+d_f$, we have 
\[
	\lvert \( \ome \cdot \nabla \)^2 \tilde{J}_n(t)f(x) \rvert \le C(1+t)^{-n/2-2}  \| f \|_{L^1}  .
\]
\end{enumerate}
\end{lem}

\section{Movement of the time-delayed hot spots}

Let $u$ denote the classical solution of the Cauchy problem \eqref{DWfg}. In this section, we investigate the asymptotic behavior of spatial maximizers of the function $u(\cdot ,t) :\R^n \to \R$.

\begin{notation}\label{assump}
{\rm 
Let us list up our notation for this section.
\begin{enumerate}
\item[($fg$)] Let $f$ and $g$ be compactly supported smooth functions such that the sum of them $h:=f+g$ is non-zero and non-negative.
\item[($\mathcal{H}$)] For a function $\phi :\R^n \to \R$, we denote by $\mathcal{M} (\phi )$ and $\mathcal{C} (\phi )$ the set of maximum and critical points of $\phi$, respectively:
\[
\mathcal{M} (\phi )= \left\{ x \in \R^n \lvert \phi (x) = \max_{\xi \in \R^n} \phi (\xi ) \right\} \right. ,\ 
\mathcal{C} (\phi )=  \left. \left\{ x \in \( \supp \phi \)^\circ \rvert \nabla \phi (x)=0 \right\} .
\]
In particular, we write
\[
\mathcal{H}(t)=\mathcal{M} \( u(\cdot,t) \)
\]
and call a point $p \in \mathcal{H}(t)$ a {\it time-delayed hot spot} at time $t$. We remark that $\mathcal{H}(t)$ is always contained in $\mathcal{C} \( u(\cdot, t) \)$. 
\item[($m$)] Under the condition $(fg)$, we investigate the distance between time-delayed hot spots and the centroid (center of mass) of $h$,
\[
	\ds m_h =\left. \int_{\R^n}h (y) y dy \right/ \int_{\R^n}h (y)dy.
\]
We remark that the centroid $m_h$ is in the interior of  the convex hull of $\supp h$.
\item[$(\de)$] Let $K$ and $L$ be convex bodies in $\R^n$. Let
\[
\de (K,L) = \sup_{\eta \in L} \dist (\eta ,K).
\]
We remark that the parallel body $K+\de (K,L)B^n$ contains the convex body $L$.
\item[($\psi$)] Let $\psi :[0,+\infty ) \to [0,+\infty )$ be a non-decreasing function such that $\psi (0)= 0$ and $\psi (t)$ is of small order of $\sqrt{t}$ as $t$ goes to infinity.
\item[($\f$)] Let $\f :[0,+\infty ) \to [0,+\infty )$ be a non-decreasing function such that $\f (0)= 0$ and $\f (t)$ is of small order of $t$ as $t$ goes to infinity.
\item[($T_0$)] Let $f$ and $g$ be as in $(fg)$. For a non-decreasing function $\phi :[0,+\infty ) \to[0,+\infty )$, let
\begin{align*}
T_0( \phi ) 
&= T_0 \( \phi ; d_h,  \de \( CS(h) ,CS(f) \) , d_f \) \\
&= \min  \left\{ T\geq 0 \lvert \forall t \geq T,\  t \geq \phi (t) + \max \left\{ d_h ,\ \de \( CS(h) ,CS(f) \) +d_f \right\} \right\} \right. .
\end{align*}
We remark that, if $t \geq T_0 (\phi)$, then, for any $x \in CS(h) +\phi (t) B^n$, the ball $B_t^n (x)$ contains the union of $CS(h)$ and $CS(f)$.
\end{enumerate}
}
\end{notation}

\subsection{Asymptotic behavior of time-delayed hot spots}

In this subsection, we are interested in the behavior of $\mathcal{C} ( u(\cdot, t) )$ and $\mathcal{H}(t)$. As we mentioned in the introduction, 
in [CK], 
Chavel and Karp showed that, for each $t$, 
the non-empty set
$H_\phi (t) := \mathcal{M} \( P_n (t)\phi \)$
is contained in the convex hull of the support of $\phi$, 
and that
\begin{equation}
\sup \left\{ \lvert x-m_\phi \rvert \lvert x\in H_\phi (t) \right\} \right. =O \( \frac{1}{t} \)
\end{equation}
as $t$ goes to infinity. Let us show that similar results hold for the damped wave equation \eqref{DWfg}.

\begin{rem}
{\rm 
Under the condition $(fg)$ in Assumption and Notation \ref{assump}, for each $t>0$, the support of $u(\cdot ,t)$ is compact. More precisely, the support of $u(\cdot ,t)$ is contained in the union of two parallel bodies $CS(h) + t B^n$ and $CS(f) +tB^n$. Hence we always have a time-delayed hot spot.
}
\end{rem}

\begin{lem}\label{hp1} 
We use Notation \ref{assump}. There exist a positive constant $C$ and a time $T \geq T_0 ( \psi )$ such that, for any $t \geq T$, we have 
\begin{align*}
	& \sup \left\{ \lvert x - m_h \rvert \lvert x\in \mathcal{C}\(u\( \cdot, t \)\) \cap  \( CS(h) +\psi (t) B^n \) \right\} \right.  \\
	&\leq C \( \frac{1+\psi (t)^2}{t}  + \frac{1+\psi (t)}{t} \frac{\| f\|_{L^1}}{\| h \|_{L^1}} + e^{-t/2}t^{3n/2} \frac{\| f\|_{W^{*,\infty}} + \| g \|_{W^{*,\infty}}}{\| h\|_{L^1}}   \) .
\end{align*}
\end{lem}

\begin{proof} 
We give a proof for even dimensional cases. The other cases go parallel.

Let $x$ be a point in $\mathcal{C}( u( \cdot, t )) \cap (CS(h) + \psi (t) B^n)$. We remark that, from Proposition \ref{prop_de_fg} and Lemma \ref{En}, we have
\[
x= \left. \( \nabla \tilde{J}_n(t) f(x) + e^{-t/2} \nabla \widetilde{\vect{W}}_n(t;f,g)(x) + \int_{\R^n} E_n(r,t)h(y)ydy \) \right/ \int_{\R^n} E_n(r,t)h(y)dy  .
\]

Since the function $E_n(\cdot ,t)$ is strictly decreasing, we obtain
\begin{align*}
&2(4\pi)^{n/2}t^{n/2+1} E_n \( \psi (t)+d_h ,t \) \| h \|_{L^1} \lvert x-m_h \rvert \\
&\leq 2(4\pi)^{n/2}t^{n/2+1} \lvert \( \int_{\R^n} E_n (r,t) h(y) dy \) \( x -m_h \) \rvert \\
&\leq 2(4\pi)^{n/2}t^{n/2+1} \( \lvert \int_{\R^n} E_n (r,t) h(y) \( y -m_h \) dy \rvert + \lvert \nabla \tilde{J}_n(t) f(x) + e^{-t/2} \nabla \widetilde{\vect{W}}_n(t;f,g)(x)  \rvert \) .
\end{align*}

Applying the third assertion in Lemma \ref{En_asymp} to the first term, there exists a constant $C$ such that, for any sufficiently large $t$, we have 
\[
2(4\pi)^{n/2}t^{n/2+1} \lvert \int_{\R^n} E_n (r,t) h(y) \( y -m_h \) dy \rvert 
\leq C \frac{1+\psi (t)^2}{t} \| h \|_{L^1} .
\]

Applying Lemmas \ref{lem_est_Wn} and \ref{lem_tildeJ} to the second term, there exists a constant $C$ such that, for any sufficiently large $t$, we have
\begin{align*}
&2(4\pi)^{n/2}t^{n/2+1}  \lvert \nabla \tilde{J}_n(t) f(x) + e^{-t/2} \nabla \widetilde{\vect{W}}_n(t;f,g)(x)  \rvert  \\
&\leq C \( \frac{1+\psi (t)}{t} \| f\|_{L^1} + e^{-t/2}t^{3n/2} \( \| f\|_{W^{n/2+1,\infty}} + \| g \|_{W^{n/2,\infty}} \) \)
\end{align*}

From the third assertion in Lemma \ref{En_asymp}, the function $2(4\pi)^{n/2}t^{n/2+1} E_n \( \psi (t)+d_h ,t \)$ is bounded from below with respect to $t$. Hence we obtain the conclusion.
\end{proof} 

\begin{cor}\label{location_critical_small}
We use Notation \ref{assump}.
\begin{enumerate}[$(1)$]
\item There exists a time $T \geq T_0 (\psi )$ such that, for any $t\geq T$, the intersection $\mathcal{C}(u(\cdot,t)) \cap ( CS(h) + \psi (t) B^n )$ is contained in the convex hull of the support of $h$.
\item We have
\[
\sup \left\{ \lvert x - m_h \rvert \lvert x\in \mathcal{C}\(u\( \cdot, t \)\) \cap  \( CS(h) + \psi (t) B^n \) \right\} \right. 
= O \( \frac{1}{t} \)
\]
as $t$ goes to infinity.
\end{enumerate}
\end{cor}

\begin{proof}
(1) Since the centroid of $h$ is in the interior of the convex hull of $\supp h$, Lemma \ref{hp1} guarantees the conclusion.

(2) From the first assertion, after a large time, we have
\[
\sup \left\{ \lvert x - m_h \rvert \lvert x\in \mathcal{C}\(u\( \cdot, t \)\) \cap  \( CS(h) + \psi (t) B^n \) \right\} \right. 
=\sup \left\{ \lvert x - m_h \rvert \lvert x\in \mathcal{C}\(u\( \cdot, t \)\) \cap  CS(h)  \right\} \right. . 
\]
Applying Lemma \ref{hp1} to the case of $\psi =0$, we get the conclusion.
\end{proof}

\begin{lem}\label{nocritical}
We use Notation \ref{assump}. Suppose that the function $\psi (t)$ diverges as $t$ goes to infinity. If $\f (t) \geq \psi (t)$, then there exists a time $T \geq T_0 (\f )$ such that, for any $t \geq T$, the gradient of $u(\cdot ,t)$ does not vanish on the region $( CS(h) + \f (t) B^n ) \sm (CS(h) + \psi (t) B^n )$.
\end{lem}

\begin{proof}
We give a proof for even dimensional cases. The other cases go parallel.

From Proposition \ref{prop_de_fg}, the solution $u(x,t)$ is expressed as
\[
u(x,t) = J_n(t) h(x) + \tilde{J}_n(t)f (x) + e^{-t/2} \widetilde{\vect{W}}_n(t;f,g)(x) .
\]
We remark that, from Lemma \ref{lem_est_Wn}, we have
\[
\lvert \nabla \widetilde{\vect{W}}_n(t;f,g)(x) \rvert \leq C (1+t)^n \( \| f \|_{W^{n/2+1, \infty}} + \| g \|_{W^{n/2, \infty}} \).
\]

From Lemma \ref{list_J} and \ref{list_tildeJ}, we have
\begin{align*}
&\nabla J_n(t) h(x) + \nabla \tilde{J}_n(t) f(x) \\
&=-\frac{c_n e^{-t/2}}{2^{n+2}} \int_{B^n_t(x)} 
\left[ tk_{\frac{n}{2}+2} \( \frac{1}{2} \sqrt{t^2 -r^2} \) f(y) +  k_{\frac{n}{2}+1} \( \frac{1}{2} \sqrt{t^2 -r^2} \) \( 4h(y) -2f(y) \) \right] (x-y) dy.
\end{align*}
Using the asymptotic expansions of $k_\ell(s)$ in Theorem \ref{thmdecom},  we have the expansion
\begin{align*}
\nabla J_n(t) h(x) + \nabla \tilde{J}_n(t) f(x) 
&=-\frac{c_n e^{-t/2}}{2^{n/2+1}} \int_{B^n_t(x)} \frac{1}{\( t^2 -r^2 \)^{n/4+1/2}} \exp \( \frac{\sqrt{t^2-r^2}}{2} \) \( 1+ O \( \frac{1}{\sqrt{t^2 -r^2}} \) \) \\
&\quad \times \( 2h(y) + \( \frac{t}{\sqrt{t^2 -r^2}} -1 \) f(y) \) (x-y) dy
\end{align*}
as $t$ goes to infinity.

In order to complete the proof, we use the contradiction argument. For any natural number $N \geq T_0 (\f )$, we assume the existence of $t_N \geq N$ such that the function $u(\cdot ,t_N )$ has a critical point $x_N$ in the region $( CS(h) + \f (t) B^n ) \sm (CS(h) + \psi (t) B^n )$. Let $r_N = \vert x_N -y \vert$. Since the unit sphere $S^{n-1}$ is compact, we may assume that the sequence $(x_N -m_h)/ \vert x_N -m_h\vert$ converges to a direction $\ome$ as $N$ goes to infinity. Then, we have
\begin{align*}
0&= -\frac{2^{n/2+1}}{c_n} \exp \( \frac{t_N - \sqrt{t_N^2 -\lvert x_N -m_h \rvert^2}}{2} \) \frac{t_N^{n/2+1}}{\lvert x_N -m_h \rvert} \nabla u\( x_N ,t_N \) \\
&=\int_{B^n_{t_N} \( x_N \)}  \exp \( \frac{\sqrt{t_N^2 -r_N^2} -\sqrt{t_N^2 -\lvert x_N -m_h \rvert^2}}{2}  \) \frac{t_N^{n/2+1}}{\( t_N^2 -r_N^2 \)^{n/4+1/2}} \( 1+ O \( \frac{1}{\sqrt{t_N^2 -r_N^2}} \) \) \\
&\quad \times  \( 2h(y) + \( \frac{t_N}{\sqrt{t_N^2 -r_N^2}} -1 \) f(y) \)  \frac{x_N-y}{\lvert x_N -m_h \rvert} dy\\
&\quad -\frac{2^{n/2+1}}{c_n} \exp \( -\frac{\sqrt{t_N^2 -\lvert x_N -m_h \rvert^2}}{2} \) \frac{t_N^{n/2+1}}{\lvert x_N -m_h \rvert} \nabla \widetilde{\vect{W}}_n \( t_N ;f,g \) \( x_N \)  \\
&\to \( 2 \int_{\R^n} h(y) dy \) \ome
\end{align*}
as $N$ goes to infinity, which contradicts to the non-negativity of $h$.
\end{proof}

\begin{cor}\label{location_critical}
We use Notation \ref{assump}. 
\begin{enumerate}[$(1)$]
\item There exists a time $T \geq T_0 ( \f )$ such that, for any $t \geq T$, the intersection $\mathcal{C} ( u(\cdot ,t) ) \cap (CS(h) + \f (t)B^n )$ is contained in the convex hull of the support of $h$.
\item We have
\[
\sup \left\{ \lvert x-m_h\rvert \lvert x \in \mathcal{C} \( u(\cdot,t) \) \cap \( CS(h) +\f (t) B^n \) \right\} \right.
= O \( \frac{1}{t} \)
\]
as $t$ goes infinity.
\end{enumerate}
\end{cor}

\begin{proof}
Corollary \ref{location_critical_small} and Lemma \ref{nocritical} guarantee the conclusion.
\end{proof}


\begin{lem}\label{expdecay}
We use Notation \ref{assump}. There exists a positive constant $C$ such that, for any $x \notin CS(h) + \f (t) B^n$ and $t>0$, we have
\[
\lvert  u(x,t) \rvert \le C \exp \( -\frac{\f (t)^2}{4t} \) \( \| h \|_{L^1} + \| f \|_{L^1} + \| f \|_{W^{* ,\infty}} + \| g \|_{W^{*, \infty}} \) .
\]
\end{lem}

\begin{proof}
We give a proof for even dimensional cases. The other cases go parallel.

From Proposition \ref{prop_de_fg}, we have
\[
u(x,t) = J_n(t) h(x) + \tilde{J}_n(t)f (x) + e^{-t/2} \widetilde{\vect{W}}_n(t;f,g) (x) .
\]

Since the function $k_\ell$ is strictly increasing, for any $x \notin CS(h) + \f (t) B^n$, $y \in \supp h$ and $t >0$, we have
\[
e^{-t/2} k_{\frac{n}{2}} \( \frac{1}{2} \sqrt{t^2 -r^2} \) 
\leq e^{-t/2} k_{\frac{n}{2}} \( \frac{1}{2} \sqrt{t^2 -\f (t)^2} \) .
\]
Using the asymptotic expansion of $k_\ell$ in Theorem \ref{thmdecom}, for any $t>0$, we have
\[
e^{-t/2} k_{\frac{n}{2}} \( \frac{1}{2} \sqrt{t^2 -\f (t)^2} \) 
\leq C \exp \( \frac{-t+\sqrt{t^2 -\f (t)^2}}{2} \) 
\leq C \exp \( -\frac{\f  (t)^2}{4t} \) .
\]
Hence, from Lemma \ref{list_J}, we can take a positive constant $C$ such that, for any $x \notin CS(h) + \f (t) B^n$ and $t>0$, we have
\[
\lvert J_n (t) h (x) \rvert \leq C \exp \( -\frac{\f (t)^2}{4t} \) \| h\|_{L^1}.
\]

In the same manner, from Lemma \ref{list_tildeJ}, we have
\[
\lvert \tilde{J}_n (t) f (x) \rvert \leq C \exp \( -\frac{\f (t)^2}{4t} \) \| f\|_{L^1}.
\]

Combining these estimates and Lemma \ref{lem_est_Wn}, we get the conclusion.
\end{proof}

\begin{lem}\label{estbelow}
We use Notation \ref{assump}.
There exist a positive constant $C$ and a time $T \geq d_h$ such that, for any $t\ge T$, we have
\[
\inf_{x\in CS(h)} u(x,t) \ge Ct^{-n/2} \| h \|_{L^\infty}.
\]
\end{lem}

\begin{proof}
We give a proof for even dimensional cases. The other cases go parallel.

We remark that, from Lemmas \ref{lem_est_Wn} and \ref{lem_tildeJ}, we have the following estimates:
\[
\lvert  \widetilde{\vect{W}}_n(t;f,g)(x) \rvert  \le C (1+t)^{n} \( \| f\|_{W^{n/2,\infty}} + \|g \|_{W^{n/2-1 ,\infty}} \) ,\ 
\lvert \tilde{J}_n(t) f(x) \rvert  \le C (1+t)^{-n/2-1} \| f \|_{L^1}.
\]

Let us estimate the function $J_n(t) h(x)$. Since $h$ is non-negative, there is a point $\eta \in CS(h)$ such that, for any $y \in B_\rho^n(\eta )$, $h(y) \geq \| h \|_{L^\infty} /2$. Using the asymptotic expansion of $k_{n/2} (s)$ in Theorem \ref{thmdecom}, we have
\begin{align*}
	&J_n(t)h(x) \\
	&\ge \frac{c_n \| h \|_{L^{\infty}}}{2^{n-1}} e^{-t/2}
		\int_{B_{\rho}^n(\eta )} k_{\frac{n}{2}} \( \frac{1}{2}\sqrt{t^2-r^2} \)dy \\
	&= \frac{c_n \| h \|_{L^{\infty}} }{2^{n/2-2}} \int_{B_{\rho}^n(\eta )}
		t^{-n/2} \( 1+ O \( \frac{1}{t^2} \) \)
		\(  1 + O\( \frac{1}{t} \) \) \( 1 + O\( \frac{1}{t} \) \) dy\\
	&\ge C t^{-n/2} \| h \|_{L^\infty}
\end{align*}
for any sufficiently large $t$.

Hence, for any sufficiently large $t$, we obtain
\begin{align*}
\lvert u(x,t) \rvert 
&\geq \lvert J_n(t) h(x) \rvert - \lvert \tilde{J}_n(t) f(x) \rvert -e^{-t/2} \lvert \widetilde{\vect{W}}_n(t;f,g)(x) \rvert \\
&\geq Ct^{-n/2}\| h \|_{L^\infty} -Ct^{-n/2-1}\| f \|_{L^1} -Ce^{-t/2} t^n \( \| f\|_{W^{n/2,\infty}} + \|g \|_{W^{n/2-1 ,\infty}} \) \\
&\geq Ct^{-n/2} \| h \|_{L^\infty} ,
\end{align*}
which completes the proof.
\end{proof}

\begin{cor}\label{location_hotspot}
We use Notation \ref{assump}. Suppose that $\exp ( -\f (t)^2 /(4t))$ is of small order of $t^{-n/2}$ as $t$ goes infinity. There exists a constant $T \geq T_0 (\f )$ such that, for any $t \geq T$, all of the time-delayed hot spots are contained in the parallel body $CS(h) + \f (t) B^n$.
\end{cor}

\begin{thm}\label{movement}
Let $f$ and $g$ be as in $(fg)$ in Notation \ref{assump}.
\begin{enumerate}[$(1)$]
\item There exists a time $T \geq \max \{ d_h ,\ \de (CS(h) , CS(f) )+ d_f \}$ such that, for any $t\geq T$, all of the time-delayed hot spots at time $t$ are contained in the convex hull of $h$.
\item We have
\[
\sup \left\{ \lvert x - m_h \rvert  \lvert x \in \mathcal{H}(t) \right\} \right. = O \( \frac{1}{t} \) 
\]
as $t$ goes to infinity.
\end{enumerate}
\end{thm}

\begin{proof}
We take a function $\f$ as in $(\f)$ in Notation \ref{assump} such that $\exp ( -\f (t)^2 /(4t))$ is of small order of $t^{-n/2}$ as $t$ goes to infinity. Then, Corollary \ref{location_hotspot} guarantees that all of the time-delayed hot spots are contained in the parallel body $CS(h)+\f (t)B^n$ after a large time. Hence, Corollary \ref{location_critical} implies the conclusion. 
\end{proof}

\begin{rem}
{\rm 
Let $g$ be a non-zero non-negative smooth function with compact support. If $n=1$ and $f=0$, then, for any $t\geq 0$, all of the time-delayed hot spots are contained in the convex hull of $\supp h= \supp g$. In other words, in this case, we can take $T=0$ in the first assertion of Theorem \ref{movement}.
}
\end{rem} 

\begin{proof}
Fix a point $x$ in the complement of the convex hull of $\supp g$. Let $x'$ be the point that gives the distance between $x$ and the convex hull of $\supp g$. We have $\vert x-y \vert > \vert x' -y \vert$ for any $y$ in $\supp g$, and $[ x-t, x+t] \cap \supp g$ is contained in $[ x'-t, x'+t] \cap \supp g$. Hence the strictly increasing behavior of $I_0$ implies
\begin{align*}
S_1 (t) g(x) 
&= \frac{e^{-t/2}}{2} \int_{x-t}^{x+t} I_0 \( \frac{1}{2} \sqrt{t^2 -\lvert x-y \rvert^2} \) g(y) dy \\
&\leq \frac{e^{-t/2}}{2} \int_{x'-t}^{x'+t} I_0 \( \frac{1}{2} \sqrt{t^2 -\lvert x-y \rvert^2} \) g(y) dy \\
&< \frac{e^{-t/2}}{2} \int_{x'-t}^{x'+t} I_0 \( \frac{1}{2} \sqrt{t^2 -\lvert x'-y \rvert^2} \) g(y) dy \\
&=S_1 (t) g \( x' \) ,
\end{align*}
which completes the proof.
\end{proof}

\subsection{Uniqueness of a time-delayed hot spot}
In [JS], 
Jimbo and Sakaguchi showed that 
the set of hot spots
$H_g(t)$
consists of one point for sufficiently large $t$.
For the damped wave equation, let us show the corresponding result to [JS].

\begin{lem}\label{concavity}
We use Notation \ref{assump}. There exists a time $T \geq \max \{ d_h, \de (CS (f) ,CS (h)) +d_f \}$ such that, for any $t\geq T$, the function $u(\cdot ,t)$ becomes strictly concave on the convex hull of the support of $h$. 
\end{lem}

\begin{proof}\normalfont
Let us give a proof for even dimensional cases. The other cases go parallel. 

In view of Proposition \ref{prop_de_fg}, we estimate the second derivatives of $J_n(t)h$, $\tilde{J}_n(t)f$ and $e^{-t/2} \widetilde{\vect{W}}_n(t;f,g)$. We fix a a point $x \in CS(h)$ and a direction $\ome \in S^{n-1}$. 

From the asymptotic expansion of $k_\ell$ in Theorem \ref{thmdecom}, there exists a positive constant $C$ such that, for any $y \in \supp h$ and $t \geq d_h$, we have
\begin{align*}
&e^{-t/2}\left[ -4k_{\frac{n}{2}+1} \( \frac{1}{2} \sqrt{t^2 -r^2} \) + k_{\frac{n}{2} +2} \( \frac{1}{2} \sqrt{t^2 -r^2} \) \( \ome \cdot (x-y ) \)^2 \right] \\
&= \frac{2^{n/2+2}}{\( t^2-r^2 \)^{n/4+1/2}} \exp \( \frac{-t+\sqrt{t^2 -r^2}}{2} \) \( -1 + O\( \frac{1}{t} \) \) \\
&\leq -C t^{-n/2-1} .
\end{align*}
Hence, from Lemma \ref{list_J}, there exists a positive constant $C$ such that, for any $x \in CS(h)$, $\ome \in S^{n-1}$ and $t \geq d_h$, we have
\[
\( \ome \cdot \nabla \)^2 J_n (t) h(x) \leq -C t^{-n/2-1} \| h \|_{L^1}.
\]

In the same manner, from Lemma \ref{list_tildeJ}, we can obtain the existence of a positive constant $C$ such that, for any $x \in CS(h)\subset CS(f)+ \de (CS(f) ,CS(h) ) B^n$, $\ome \in S^{n-1}$ and $t \geq  \de (CS(f) ,CS(h)) +d_f$, we have
\[
\lvert \( \ome \cdot \nabla \)^2 \tilde{J}_n(t) f(x) \rvert \leq Ct^{-n/2 -2} \| f\|_{L^1}. 
\]

On the other hand, from Lemma \ref{lem_est_Wn}, there exists a positive constant $C$ such that, for any $x \in CS(h)$, $\ome \in S^{n-1}$ and $t>0$, we have
\[
e^{-t/2} \lvert \( \ome \cdot \nabla \)^2 \widetilde{\vect{W}}_n(t; f,g) (x) \rvert \leq Ce^{-t/2}(1+t)^{n} \( \| f \|_{W^{n/2+2,\infty}} + \| g \|_{W^{n/2+1 ,\infty}} \).
\]

Hence  we obtain the strict concavity of $u(\cdot ,t)$ on $CS(h)$ for any sufficiently large $t$.
\end{proof}

\begin{prop}\label{uniqueness}
We use Notation \ref{assump}. There exists a time $T \geq \max \{ T_0 (\f), \de (CS (f) ,CS (h)) +d_f \}$ such that, for any $t\ge T$, the set of critical points of $u(\cdot ,t)$ contained in the parallel body $CS(h) + \f (t) B^n$ consists of one-point.
\end{prop}

\begin{proof}
From Corollary \ref{location_critical}, it is sufficient to show the strict concavity of the function $u(\cdot ,t)$ on the convex hull of the support of $h$. Hence Lemma \ref{concavity} guarantees the uniqueness of a critical point of $u(\cdot ,t)$.
\end{proof}

\subsection{Wave effect of the damped wave in view of time-delayed hot spots}

In this subsection, we investigate the wave properties of the damped wave equation in view of the movement of time-delayed hot spots.
We give some examples of initial data $(f,g)$ which allow time-delayed hot spots to escape from the convex hull of the support of $h:=f+g$ for some small time.

\begin{ex}\label{ex1}
{\rm Let $n=1$. We consider the equation \eqref{DWfg} with $g=0$.
By Example \ref{decom1} and Proposition \ref{prop_de_fg}, we have
\begin{align*}
u(x,t)
&=S_1(t)f(x)+ \tilde{J}_1(t) f(x) +e^{-t/2} \widehat{\vect{W}}_n(t) f(x) \\
&=\frac{e^{-t/2}}{4} \int_{x-t}^{x+t} \( \frac{t}{\sqrt{t^2 -r^2}} I_1 \( \frac{1}{2} \sqrt{t^2 -r^2} \) +I_0 \( \frac{1}{2} \sqrt{t^2 -r^2} \) \) f(y) dy \\
&\quad +\frac{e^{-t/2}}{2}\( f(x+t)+f(x-t)\)
\end{align*}
Let us give an example such that if the initial datum $f$ has a sufficiently large maximum value and a constant $L^1$ norm, then, for some $t$, time-delayed hot spots escape from the convex hull of the support of $f$.

Let
$\rho$ be a non-negative smooth function satisfying
$\supp \rho = [-2, 2]$, $\| \rho \|_{L^1}=1$ and
$\rho(y)\ge \| \rho \|_{L^{\infty}}/2$ for $-1 \leq y \leq 1$.
For example, if we normalize the function
\[
	\tilde{\rho}(y) = 
	\begin{cases}
	\ds \exp \( -\frac{1}{4-y^2} \) &(-2 \leq y \leq 2),\\
	0                               &({\rm otherwise}),
	\end{cases}
\]
then the normalized function
$\tilde{\rho}/\| \tilde{\rho}\|_{L^1}$
satisfies the conditions of $\rho$. We define
$f_\ep (y)=\rho(y/\ep)/\ep$
with a small parameter
$\ep>0$.
Then we have
$f_\ep (x+t)\geq \|\rho\|_{L^{\infty}}/(2\ep )$
for
$-t-\ep \leq x \leq -t+\ep$
and
$f_\ep (x-t) \geq \|\rho\|_{L^{\infty}}/(2\ep)$
for
$t-\ep \leq x \leq  t+\ep$.

On the other hand,
noting
$\| f_\ep \|_{L^1}=1$,
we can choose a constant
$C$ independent of $\ep$ such that we have
\[
\| S_1 (t)f+ \tilde{J}_1(t) f \|_{L^\infty} \leq C e^{-t/2}\( I_0 \(\frac{t}{2} \) + (1+t) \( 1+ I_1 \( \frac{t}{2} \) \) \) .
\]

Using the facts $I_0(0)=1$ and $I_1(0)=0$, we can take the parameter $\ep$ sufficiently small so that there is a time $t\geq 4\ep$ satisfying the inequality
\[
	\frac{e^{-t/2}}{2}\frac{\| f_\ep \|_{L^\infty}}{2} 
	=\frac{e^{-t/2}}{2}\frac{\| \rho \|_{L^\infty}}{2\ep} 
	> Ce^{-t/2} \( I_0\(\frac{t}{2} \)+ (1+t) \( 1+ I_1 \( \frac{t}{2}\) \)\) .
\] 
Hence if  $t\geq 4\ep$ satisfies the above inequality and
$x \in [t-\ep, t+\ep]\cup [-t-\ep, -t+\ep]$
then we have $u(x,t) > u(\xi ,t)$ for
any $\xi \in \supp f_\ep$.
}
\end{ex}

\begin{ex}\label{ex2} 
{\rm
Let $n=2$. We consider the damped wave equation \eqref{DWfg} with $f=0$.
By Proposition \ref{propsol}, we have
\[
u(x,t)=S_2(t)g(x)
=\frac{e^{-t/2}}{2\pi}\int_{B_t^2(x)} \frac{\cosh(\frac{1}{2}\sqrt{t^2-r^2})}{\sqrt{t^2-r^2}}g(y)dy.
\]
Let us show that, if we choose a clever initial datum $g$, then, for some $t$, $S_2(t)g$ has a (non-trivial) critical point in the complement of the convex hull of $\supp g$.

Let $s_*$ be the unique critical point of the function $\cosh(s/2)/s$. Direct computation shows $2<s_*<3$.
Fix a small positive parameter $\ep$ with $2 \ep < s_*$. Then, we have $s_* < (s_*^2 +4\ep^2)/(4\ep)$.
Let $g_\ep$ be a non-zero non-negative radially symmetric smooth function with support $B_{\ep}^2(0)$. Let us show that, for any $s_* \leq t \leq (s_*^2 +4\ep^2)/(4\ep)$, the function $S_2(t)g$ has a critical point in the complement of the disk of radius $\sqrt{t^2 -s_*^2}+\ep$ centered at origin. 

Fix a point $x$ with $\vert x \vert = \sqrt{t^2 -s_*^2}+\ep$. We remark that, for any $y \in B_\ep^2 (0)$, we have the following inequalities:
\[
\lvert x-y \rvert < t ,\ \sqrt{t^2 - \lvert x -y \rvert^2} \leq s_* .
\]
Let
\[
\de = \frac{t- \sqrt{t^2 -s_*^2}-2\ep}{2} \frac{1}{\sqrt{t^2 -s_*^2} +\ep}.
\]
Then, for the point $x':= (1+\de )x$ and any $y \in B_\ep^2(0)$, we have the following inequalities:
\[
\lvert x'-y \rvert <t,\ \sqrt{t^2 - \lvert x' -y \rvert^2} \leq s_*.
\]
Therefore, we have
\begin{align*}
	S_2(t)g_\ep (x)&=\frac{e^{-t/2}}{2\pi}\int_{B_t^2(x)}
		\frac{\cosh \( \frac{1}{2}\sqrt{t^2-|x-y|^2} \)}{\sqrt{t^2-|x-y|^2}}g_\ep(y)dy\\
	&< \frac{e^{-t/2}}{2\pi}\int_{B_t^2(x')}
		\frac{\cosh \( \frac{1}{2}\sqrt{t^2-\lvert x' -y\rvert^2} \)}{\sqrt{t^2-\lvert x'-y\rvert^2}}g_\ep (y)dy\\
	&=S_2(t)g_\ep \( x'\) .
\end{align*}
Thanks to the compactness of the support of $S_2(t)g_\ep$, for each direction $\ome \in S^1$, we get the existence of a maximal point of the function
\[
\( \sqrt{t^2 -s_*^2}+\ep ,t+\ep \) \ni \rho \mapsto S_2(t) g_\ep (\rho \ome ) \in \R.
\] 
Since the function $g_\ep$ is radially symmetric, the function $S_2(t) g_\ep$ so is, and we obtain the existence of a critical point of $S_2(t) g_\ep$ in the complement of the disk of radius $\sqrt{t^2 -s_*^2}+\ep$ centered at origin. 
}
\end{ex}

\begin{ex}\label{ex3} 
{\rm Let $n=2$. We consider the damped wave equation \eqref{DWfg} with $f=0$ again. 
Let $g$ be a non-zero non-negative smooth function with compact support.
Suppose $2 d_g <s_*$. 
Let us show that, for any $2d_g \leq t \leq s_*$, 
there exists a point $x$ in the complement of $CS(g)$ such that, 
for any point $\xi \in CS( g)$, we have $S_2(t) g(\xi ) < S_2 (t) g(x)$.
In other words, if $g$ has a small support so that $d_g <s_* /2$, then, for any $2d_g \leq t \leq s_*$, time-delayed hot spots escape from the convex hull of the support of $g$. 

Fix an arbitrary time $2d_g \leq t \leq s_*$. We can choose a point $x \in CS( g)^c$ which satisfies the following conditions:
\[
\max_{y\in \supp g} \lvert x-y \rvert =t,\ \min_{y \in \supp g} \lvert x-y \rvert \geq t-d_g .
\]
For such a point $x$, any $\xi \in CS( g)$ and $y \in \supp g$, the assumption of $t$ implies 
\[
0 \leq \sqrt{t^2 -\lvert x-y \rvert^2} \leq \sqrt{\(2t-d_g\)d_g} \leq \sqrt{t^2 -d_g^2} \leq \sqrt{t^2 -\lvert \xi -y \rvert^2} \leq t,
\]
and the strictly decreasing behavior of $\cosh(s/2)/s$ for $0<s<s_*$ implies
\[
\frac{\cosh \( \frac{1}{2} \sqrt{t^2 -\lvert x-y \rvert^2}\)}{\sqrt{t^2 -\lvert x-y \rvert^2}} \geq \frac{\cosh \( \frac{1}{2} \sqrt{t^2 -\lvert \xi-y \rvert^2}\)}{\sqrt{t^2 -\lvert \xi-y \rvert^2}}.
\]
Hence  we obtain $S_2(t) g(x) > S_2(t) g(\xi )$ for any $\xi \in CS( g)$.
}
\end{ex}

\begin{ex}\label{ex4}
{\rm 
Let $n=3$. We consider the damped wave equation \eqref{DWfg} with $f=0$.  
By Example \ref{decom23}, we have
\[
	S_3(t)g =J_3(t)g +e^{-t/2}W_3(t)g .
\]
Let us give an example such that if the initial datum $g$ has a sufficiently large maximum value and a constant $L^1$ norm, then, for some $t$, hot spots escape from the convex hull of the support of $g$.

Let $\rho$ be non-zero non-negative smooth function satisfying
$\supp \rho = B_2^3(0)$, $\| \rho \|_{L^1}=1$ and $\rho (y)\ge \| \rho \|_{L^\infty}/2$ on the unit ball $B^3$.
We define
$g_\ep (y)=\rho \( y/\ep \) /\ep^3$
with a small parameter $\ep >0$. 

If  $t>2\ep$ and $x \in S_t^2(0)$, then we have
\[
	\sigma_2\( S_t^2(x) \cap B^3_{\ep}(0) \) =\pi \ep^2, 
\]
and then, we get
\[
	e^{-t/2}W_3(t)g_\ep(x)
	=\frac{e^{-t/2}}{4\pi t}\int_{S_t^2(x)}g_\ep (y)d\sigma_2(y)
	\ge \frac{\ep^2 e^{-t/2}}{8t}\| g_\ep\|_{L^\infty}
	=\frac{e^{-t/2}}{8\ep t} \| \rho \|_{L^\infty}.
\]

On the other hand, as we will see in \eqref{4_lem2_eq1},
$J_3(t)g_{\ep}$
is estimated by
\[
	\| J_3(t)g_\ep \|_{L^\infty} 
	\leq C(1+t)^{-3/2}\|g_\ep \|_{L^1}
	=C(1+t)^{-3/2},
\]
where $C$ is independent of $\ep$.

We can take the parameter $\ep$ sufficiently small so that there is a time $t\geq 4\ep$ satisfying the inequality
\[
\frac{e^{-t/2}}{8\ep t}\| \rho \|_{L^\infty} >C(1+t)^{-3/2} .
\]
If $t \geq 4\ep$ satisfies the above inequality, then, for any $x \in S_t^2(0)$ and $\xi \in \supp g_\ep =B_{2\ep}^3(0)$, we have
\[
S_3(t) g_\ep (x) \geq e^{-t/2}W_3(t)g_\ep (x) > J_3(t)g_\ep (\xi ) = S_3(t) g_\ep (\xi),
\]
that is, time-delayed hot spots are not in (the convex hull of) the support of $g_\ep$.
}
\end{ex}

\section{Application of the Nishihara decomposition: {\boldmath $L^p$-$L^q$} estimates}
In this section, as an application of Theorem \ref{thmdecom},
we give
$L^p$-$L^q$ estimates for the solution of the damped wave equation \eqref{DWfg}.
In [HO, MN, Nis],
when $n\leq 3$, 
the following $L^p$-$L^q$ estimates were shown:
\begin{equation}
\label{4_LpLq}
	\left\| u(\cdot ,t)-P_n(t) \( f+g \) -e^{-t/2}\widetilde{\vect{W}}_n(t;f,g)\right\|_{L^p}
	\leq Ct^{-\frac{n}{2} \( \frac{1}{q}-\frac{1}{p} \) -1} \( \|f\|_{L^q}+\|g\|_{L^q} \),
\end{equation}
where $t>0$ and $1\leq q\leq p\leq \infty$.
In [Nar], when $n\ge 4$,
Narazaki showed
the following estimates:
\begin{align}
\label{4_Na}
	\left\|\mathcal{F}^{-1}\left[ \(\hat{u}(\cdot ,t)-\hat{v}(\cdot ,t)\) \chi \right] \right\|_{L^p}
	\le C \( 1+t \)^{-\frac{n}{2} \( \frac{1}{q}-\frac{1}{p} \) -1+\ep}
		\( \|f\|_{L^q}+\|g\|_{L^q} \),
\end{align}
where
$1\leq q\leq p\leq \infty$,
$\ep$
is an arbitrary small positive number,
$C=C(n,p,q,\ep)$
is a positive constant,
$\chi$
is a compactly supported radially symmetric smooth function satisfying
$\chi=1$
near the origin,
$v(x,t)=P_n(t)(f+g)(x)$,
$\hat{u}$ and $\hat{v}$
denote the Fourier transform of
$u$ and $v$,
respectively,
and
$\mathcal{F}^{-1}$
is the inverse Fourier transform.
Moreover, in the case where
$1<q<p<\infty$, $(p,q)=(2,2)$
or
$(p,q)=(\infty, 1)$,
we may take
$\ep=0$, that is, we have
\begin{equation}
\label{Nar_est2}
	\left\| \mathcal{F}^{-1}\left[ \( 1-\chi \) \( \hat{u}(\cdot, t)
	-e^{-t/2} \( \vect{M}_0(\cdot ,t )\hat{f}(\cdot ,t)+\vect{M}_1(\cdot ,t) \hat{g}(\cdot ,t) \) \) \right]
		\right\|_{L^p}
	\leq Ce^{-\de t} \|g\|_{L^q}
\end{equation}
for some
$\delta>0$,
where
$1<q\leq p<\infty$,
$C=C(n,p,q)$
is a positive constant,
and
\begin{align}
\nonumber
\vect{M}_1(\xi ,t)
&= \frac{1}{\sqrt{|\xi|^2-1/4}}\( \sin \( t|\xi| \) \sum_{0\le k<(n-1)/4}\frac{(-1)^k}{(2k)!}t^{2k}\Theta(\xi)^{2k} \right. \\
&\quad-\left. \cos \( t|\xi|\) \sum_{0\le k<(n-3)/4}\frac{(-1)^k}{(2k+1)!} t^{2k+1}\Theta(\xi)^{2k+1}\),\\
\nonumber
\vect{M}_0(\xi ,t)
&=\cos \( t|\xi| \) \sum_{0\le k<(n+1)/4}\frac{(-1)^k}{(2k)!}t^{2k}\Theta(\xi)^{2k}\\
&\quad +\sin \( t|\xi| \) \sum_{0\le k<(n-1)/4}\frac{(-1)^k}{(2k+1)!} t^{2k+1}\Theta(\xi)^{2k+1}+\frac{1}{2}\vect{M}_1(\xi ,t)
\end{align}
with
$\Theta(\xi)=|\xi|-\sqrt{|\xi|^2-1/4}$.

The aim of this section is to remove
the $\ep$
in the estimate \eqref{4_Na} and the restriction $q\neq 1$ and $p \neq \infty$.

\begin{thm}\label{4_thm2}
Let $1\leq q \leq p \leq \infty$.
Assume that the initial data $f$ and $g$ are $L^q$-integrable smooth functions.
Let $u$ be the solution to \eqref{DWfg}. There exists a positive constant $C$ such that, for any $t>0$, we have
\[
	\left\| u(\cdot, t)-P_n(t) \( f+g \) -e^{-t/2}\widetilde{\vect{W}}_n \( t;f,g\) \right\|_{L^p}
	\leq Ct^{-\frac{n}{2} \( \frac{1}{q}-\frac{1}{p} \) -1} \( \|f\|_{L^q}+\|g\|_{L^q} \) .
\]
\end{thm}

The proof of this theorem is almost same as in [Nis]
(see also [INZ]).
We note that the one-dimensional case has already been proved by
Marcati and Nishihara in [MN], and we give a proof only for the case $n\ge 2$.
By Proposition \ref{prop_de_fg}, we have
\[
	u(\cdot ,t)-P_n(t) \( f+g \) -e^{-t/2}\widetilde{\vect{W}}_n \( t;f,g\)
	=J_n(t) \( f+g \) -P_n(t) \( f+g \) +\tilde{J}_n(t)f.
\]
Therefore, the proof is reduced to the following estimates:
\begin{lem}\label{4_lem2}
Let $1\le q\le p\le \infty$, and $g$ an $L^q$-integrable smooth function.
There exists a constant
$C>0$
such that, for any $t >0$, we have the following inequalities:
\begin{align}
\label{4_lem2_eq1}
	\left\| J_n(t)g \right\|_{L^p}
		&\le C \( 1+t \)^{-\frac{n}{2}\(\frac{1}{q}-\frac{1}{p}\)}\|g\|_{L^q},\\
\label{4_lem2_eq2}
	\left\| \tilde{J}_n(t)g \right\|_{L^p}
		&\le C(1+t)^{-\frac{n}{2}\(\frac{1}{q}-\frac{1}{p}\)-1}\|g\|_{L^q} ,\\
\label{4_lem2_eq3}
	\left\| J_n(t)g-P_n(t)g \right\|_{L^p}
		&\leq Ct^{-\frac{n}{2}\(\frac{1}{q}-\frac{1}{p}\)-1}\|g\|_{L^q} .
\end{align}
\end{lem}

\begin{proof} 
Let us give a proof for higher odd dimensional cases. Even dimensional cases go parallel.

We first show \eqref{4_lem2_eq1} and \eqref{4_lem2_eq3}. We assume $t\ge 1$ and write $\tilde{c}_n = 2^{-(n-1)}c_n$.
For a constant $0 < \ep <1/2$, put
\begin{align*}
	X_1&=\int_{t^{(1+\ep)/2}B^n(x)}
		\( \tilde{c}_n e^{-t/2}
			k_{\frac{n-1}{2}} \( \frac{1}{2} \sqrt{t^2-r^2} \)
				-\frac{e^{-r^2/(4t)}}{(4\pi t)^{n/2}}
		\)g(y)dy,\\
	X_2&=\int_{B_t^n(x)\sm t^{(1+\ep)/2}B^n(x)}
		\( \tilde{c}_n e^{-t/2}
			k_{\frac{n-1}{2}} \(\frac{1}{2}\sqrt{t^2-r^2}\)
				-\frac{e^{-r^2/(4t)}}{(4\pi t)^{n/2}}
		\)g(y)dy,\\
	X_3&=\int_{B_t^n(x)^c}\frac{e^{-r^2/(4t)}}{(4\pi t)^{n/2}}g(y)dy.
\end{align*}
Then we have
\[
	J_n(t)g(x) -P_n(t)g (x) =X_1+X_2+X_3.
\]

By the Hausdorff-Young inequality ([GGS, p. 142]),
we estimate the integral $X_3$ as 
\[
	\|X_3\|_{L^p}
	\le \(\int_{B_t^n(0)^c}\frac{e^{-\rho|y|^2/(4t)}}{(4\pi t)^{\rho n/2}}dy\)^{1/\rho}
	\|g\|_{L^q}
	\le e^{-t/8}\|g\|_{L^q},
\]
where
$\rho$
is determined by the relation
$1/q-1/p=1-1/\rho$.

In the same manner as in the above estimate, we can obtain
\[
	\|X_2\|_{L^p}
	\le e^{-ct^\ep}\|g\|_{L^q}
\]
with some constant
$c>0$.

Let us estimate the integral $X_1$.
By the asymptotic expansion in Theorem \ref{thmdecom},
we have
\begin{align*}
	&\tilde{c}_n e^{-t/2}
			k_{\frac{n-1}{2}}\(\frac{1}{2}\sqrt{t^2-r^2}\)\\
	&=\frac{1}{(4\pi)^{n/2}}\frac{1}{\( t^2-r^2 \)^{n/4}}
		\exp\( \frac{-t+ \sqrt{t^2-r^2}}{2} \)
		\(1-\frac{n(n-2)}{4\sqrt{t^2-r^2}}+O\(\frac{1}{t^2-r^2}\)\).
\end{align*}
Therefore, we obtain
\[
	X_1=\frac{1}{(4\pi t)^{n/2}}
	\int_{t^{(1+\ep)/2}B^n(x)}e^{-r^2/(4t)}F(r,t)g(y)dy,
\]
where
\[
	F(r,t)=\exp\( \frac{r^2}{4t}+\frac{-t+\sqrt{t^2-r^2}}{2} \)
		\(\frac{t}{\sqrt{t^2-r^2}}\)^{n/2}
		\(1-\frac{n(n-2)}{4\sqrt{t^2-r^2}}
		+O\(\frac{1}{t^2-r^2}\)\)
		-1.
\]
Hence we have
\[
	\|X_1\|_{L^p}\le \frac{C}{t^{n/2}}
		\(\int_{t^{(1+\ep)/2}B^n} e^{-\rho r^2/(4t)}
			\lvert F \( |y|,t \)  \rvert^{\rho} dy\)^{1/\rho}
		\|g\|_{L^q}
\]
with
$1/q-1/p=1-1/\rho$.
Asymptotic expansions \eqref{app_taylor1} and \eqref{app_taylor3} imply
\begin{align*}
	F \( |y|,t \) &=\(1+\frac{1}{t}O\(\frac{|y|^4}{t^2}\)\)
		\(1+\frac{1}{t}O\(\frac{|y|^2}{t}\)\)^{n/2} 
		 \(1+ O \( \frac{1}{t} \) +\frac{1}{t} O\( \frac{\lvert y \rvert^2}{t}\) \)-1\\
	&=\frac{1}{t}O\(1+\frac{|y|^2}{t}+\cdots+\(\frac{|y|^2}{t}\)^N\)
\end{align*}
for some large integer
$N$.
Consequently, we obtain
\begin{align*}
	\|X_1\|_{L^p}&\le \frac{C}{t^{n/2+1}}
		\(\int_{t^{(1+\ep)/2}B^n(0)}e^{-\rho |y|^2/(4t)}
			\(1+\frac{|y|^2}{t}+\cdots+\(\frac{|y|^2}{t}\)^N\)^{\rho}
				dy\)^{1/\rho} \|g\|_{L^q}\\
	&\le \frac{C}{t^{n/2+1}}t^{n/(2\rho)}
		\(\int_{\R^n}e^{-\rho |z|^2}
			\( 1+|z|^2+\cdots+|z|^{2N} \)^{\rho} dz \)^{1/\rho} \|g\|_{L^q}\\
	&\le Ct^{-\frac{n}{2} \( \frac{1}{q}-\frac{1}{p} \) -1}\|g\|_{L^q},
\end{align*}
which implies the estimate \eqref{4_lem2_eq3} for
$t\ge 1$.
Moreover, we recall the well-known fact
\[
	\|P_n(t)g\|_{L^p}\le Ct^{-\frac{n}{2}(\frac{1}{q}-\frac{1}{p})}\|g\|_{L^q},\ t>0
\]
(see [GGS, p. 8]).
Using this fact, we have
\[
	\left\| J_n(t)g \right\|_{L^p}\leq
	 \left\| J_n(t)g-P_n(t)g \right\|_{L^p} + \left\| P_n(t)g \right\|_{L^p}
	\le Ct^{-\frac{n}{2} \( \frac{1}{q}-\frac{1}{p} \) }\|g\|_{L^q},
\]
which implies \eqref{4_lem2_eq1} for $t\ge 1$.
The estimates \eqref{4_lem2_eq1} and \eqref{4_lem2_eq3} for 
$0\le t<1$
are easy, and we omit the proof.

Next, we show the estimate \eqref{4_lem2_eq2}.
We assume $t\ge 1$. 
For a constant $0<\ep <1/2$, put
\begin{align*}
	X_4&=\tilde{c}_n\int_{B_t^n(x)\sm t^{(1+\ep)/2}B^n(x)}
		\frac{\pd}{\pd t}\(e^{-t/2}k_{\frac{n-1}{2}}\(\frac{1}{2}\sqrt{t^2-r^2}\)\) g(y)dy,\\
	X_5&=\tilde{c}_n\int_{t^{(1+\ep)/2}B^n(x)}
		\frac{\pd}{\pd t} \( e^{-t/2}k_{\frac{n-1}{2}} \( \frac{1}{2}\sqrt{t^2-r^2}\)\)g(y)dy.
\end{align*}
Then we have
\[
\tilde{J}_n(t)g(x) =X_4+X_5.
\]

Since
$k_{\ell+1}(s)=k'_{\ell}(s)/s$ leads to
\begin{align*}
	\frac{\pd}{\pd t} \( e^{-t/2}k_{\frac{n-1}{2}}\(\frac{1}{2}\sqrt{t^2-r^2}\)\)
	&=e^{-t/2}\left[
		-\frac{1}{2}k_{\frac{n-1}{2}}\(\frac{1}{2}\sqrt{t^2-r^2}\)
		+\frac{t}{2\sqrt{t^2-r^2}}
			k_{\frac{n-1}{2}}' \(\frac{1}{2}\sqrt{t^2-r^2}\)\right] \\
	&=e^{-t/2}\left[
		-\frac{1}{2}k_{\frac{n-1}{2}}\(\frac{1}{2}\sqrt{t^2-r^2}\)
		+\frac{t}{4}k_{\frac{n+1}{2}}\(\frac{1}{2}\sqrt{t^2-r^2}\) \right],
\end{align*}
in the same manner as the estimate of $X_2$,
we can obtain
\[
	\|X_4\|_{L^p}\leq Ce^{-ct^\ep}\|g\|_{L^q}
\]
with some constant
$c>0$.

Now we turn to the estimate for
$X_5$.
By using the asymptotic expansion of $k_\ell$, \eqref{app_taylor1} and \eqref{app_taylor2} again,
we have
\begin{align*}
	&\frac{\pd}{\pd t}\(e^{-t/2}k_{\frac{n-1}{2}}\(\frac{1}{2}\sqrt{t^2-r^2}\)\)\\
	&=\frac{2^{(n-3)/2}}{\sqrt{\pi}}
		e^{-r^2/(4t)}
		\exp\( \frac{r^2}{4t}+\frac{-t+\sqrt{t^2-r^2}}{2} \)
		(t^2-r^2)^{-n/4}\(\frac{t}{\sqrt{t^2-r^2}}-1\)
		\( 1+O\( \frac{1}{t} \) \) \\
	&=\frac{2^{(n-3)/2}}{\sqrt{\pi}}
		e^{-r^2/(4t)}
		\( 1+\frac{1}{t}O\(\frac{r^2}{t}\)\)
		t^{-n/2}\( 1+\frac{1}{t}O\(\frac{r^2}{t}\)\)^{n/2}
		\frac{1}{t}O\(\frac{r^2}{t}\)
		\(1+O\(\frac{1}{t}\)\)\\
	&\leq Ct^{-n/2-1}e^{-r^2(4t)}
		\( 1+\frac{r^2}{t}+\cdots+\(\frac{r^2}{t}\)^N\)
\end{align*}
on
the ball $t^{(1+\ep)/2}B^n(x)$
with some large integer
$N$.
Consequently, we obtain
\begin{align*}
	\|X_5\|_{L^p}&\le Ct^{-\frac{n}{2}-1}
	\(\int_{t^{(1+\ep)/2}B^n}e^{-\rho |y|^2/(4t)}
		\(1+\frac{|y|^2}{t}+\cdots+\(\frac{|y|^2}{t}\)^N\)^{\rho} dy \)^{1/\rho}
		\|g\|_{L^q}\\
	&\le Ct^{-\frac{n}{2}-1-\frac{n}{2\rho}}
		\(\int_{\R^n} e^{-\rho |z|^2/4}
			\( 1+|z|^2+\cdots+|z|^{2N} \)^{\rho} dz \)^{1/\rho} \|g\|_{L^q}\\
	&\le Ct^{-\frac{n}{2}\(\frac{1}{q}-\frac{1}{p}\)-1}\|g\|_{L^q}
\end{align*}
with
$1/q-1/p=1-1/\rho$,
which implies \eqref{4_lem2_eq2} for
$t\ge 1$.

The estimate \eqref{4_lem2_eq2} for
$0\le t<1$
is easy, and we omit the proof.
\end{proof}

\section{Appendices}
\subsection{Proofs of preliminary estimates}
\begin{proof}[Proof of Lemma \ref{lem_est_Wn}]
Let us give a proof for even dimensional cases. The other cases go parallel. 

Changing the variable as $y=x+tz$ with $z \in B^n$, we have
\[
	\int_{B_t^n(x)} \frac{1}{\sqrt{t^2-r^2}} g(y) dy
	= t^{n-1} \int_{B^n} \frac{1}{\sqrt{1-\lvert z \rvert^2}} g(x+tz) dz.
\]

When we estimate the function 
\[
\vect{W}_n(t)g(x)
=2c_n \sum_{j=0}^{(n-2)/2} \frac{1}{8^j j !}\( \frac{1}{t} \frac{\pd}{\pd t} \)^{(n-2)/2-j} \(  t^{n-1} \int_{B^n} \frac{1}{\sqrt{1-\lvert z \rvert^2}} g(x+tz) dz \) ,
\]
the worst term with respect to the growth order of $t$ is given by $j= (n-2)/2$. We can bound it above by $C(1+t)^{n-1} \| g \|_{L^\infty}$. Furthermore, we can bound the term of $j =0$ above by  $C (1+t)^{n/2} \| g \|_{W^{n/2-1 ,\infty}}$. The other terms are bounded above by these two quantities (up to a constant multiple).  Hence we obtain the estimate for $\vect{W}_n(t)g(x)$.

From Proposition \ref{prop_de_fg}, we have
\[
\widehat{\vect{W}}_n(t) f(x) = \frac{c_n t^n}{2^{\frac{3n-2}{2}}\(\frac{n}{2} \) !} \int_{B^n} \frac{1}{\sqrt{1-\lvert z \rvert^2}} f(x+tz) dz,
\]
which implies the estimate for $\widehat{\vect{W}}_n(t) f(x)$. Also, we have
\[
\widetilde{\vect{W}}_n(t;f,g) (x)
= \frac{1}{2} \vect{W}_n (t) f(x) + \vect{W}_n (t) g(x) + \widehat{\vect{W}}_n(t) f(x)  + \frac{\pd}{\pd t} \vect{W}_n(t)f(x).
\]
Combining the above estimates for $\vect{W}_n(t)g(x)$ and $\widehat{\vect{W}}_n(t) f(x)$, we obtain the conclusion.
\end{proof}

\begin{proof}[Proof of Lemma \ref{list_J}]
Using Remark \ref{recursion}, integration by parts implies the identities.
\end{proof}

\begin{proof}[Proof of Lemma \ref{En}] 
If $\dist (x, CS (h) ) \leq t-d_h$ and $t\geq d_h$, then the intersection $S^{n-1}_t(x) \cap CS(h)$ is a null set with respect to the $(n-1)$-dimensional spherical Lebesgue measure. Hence Lemma \ref{list_J} guarantees the conclusion.
\end{proof} 

\begin{proof}[Proof of Lemma \ref{En_asymp}] 
(1) We give a proof for even dimensional cases. The other cases go parallel.

We remark that, from the definition of $c_n$ \eqref{cn}, we have
\[
E_n (r,t) =  \frac{e^{-t/2}}{2^{3n/2+1}\pi^{n/2}} k_{\frac{n}{2}+1}\( \frac{1}{2}\sqrt{t^2-r^2} \) .
\]
From Theorem \ref{thmdecom}, we have
\[
k_{\frac{n}{2}+1} \( \frac{1}{2} \sqrt{t^2 -\f (t)^2} \) 
= \frac{1}{2} \( \frac{2}{\sqrt{t^2 -\f (t)^2}} \)^{n/2+1} \exp \( \frac{\sqrt{t^2 -\f (t)^2}}{2} \) \( 1+ O \( \frac{1}{\sqrt{t^2 -\f (t)^2}} \) \)
\]
as $t$ goes to infinity, which implies the conclusion.

(2) Applying the fact \eqref{app_taylor1} to the first assertion, we obtain the conclusion.

(3) Applying the fact \eqref{app_taylor2} to the second assertion, we obtain the conclusion.
\end{proof}

\begin{proof}[Proof of Lemma \ref{list_tildeJ}]
Using Remark \ref{recursion}, integration by parts implies the identities.
\end{proof}

\begin{proof}[Proof of Lemma \ref{lem_tildeJ}]
(1) This is a direct consequence of \eqref{4_lem2_eq2} in Lemma \ref{4_lem2}.

(2) We give a proof for even dimensional cases. The other cases go parallel.

From the asymptotic expansion in Theorem \ref{thmdecom}, we have
\begin{align*}
&e^{-t/2} \( t k_{\frac{n}{2}+2} \( \frac{1}{2} \sqrt{t^2 -r^2} \) -2 k_{\frac{n}{2}+1} \( \frac{1}{2} \sqrt{t^2 -r^2} \) \) \\
&=\frac{2^{n/2+1}}{\( t^2 -r^2 \)^{n/4+1/2}} \exp \( \frac{-t+\sqrt{t^2 -r^2}}{2} \) \\
&\quad \times \left[ \frac{t}{\sqrt{t^2 -r^2}} \( 1+ O \( \frac{1}{\sqrt{t^2-r^2}} \) \) - \( 1+ O \( \frac{1}{\sqrt{t^2-r^2}} \) \)  \right] .
\end{align*}
Using the facts \eqref{app_taylor1} and \eqref{app_taylor2}, the above expansion coincides with
\[
2^{n/2+1} t^{-n/2-1} \( O \( \frac{1}{t} \) + O \( \frac{\psi (t)^2}{t^2} \) \) ,
\]
which implies the conclusion.

(3) Using the asymptotic expansion in Theorem \ref{thmdecom}, in the same manner as in the second assertion, we obtain the conclusion.
\end{proof}
\subsection{Frequently used Taylor's expansions}
Let us list up frequently used Taylor's expansions:
\begin{itemize}
\item As $s$ tends to zero, we have
\begin{equation}
\label{app_taylor}
(1-s^2)^\al = 1  - \al s^2 + \frac{\al (\al-1)}{2} s^4 + O \( s^6 \) .
\end{equation}
\item If a function $\f (t)$ is of small order of $t$ as $t$ goes to infinity, then we have 
\begin{equation}
\label{app_taylor1}
	\( t^2-\f (t)^2 \)^\al = t^{2\al}
		\( 1 -\al \( \frac{\f (t)}{t} \)^2 + \frac{\al (\al-1)}{2} \( \frac{\f (t)}{t} \)^4 
		+ O\( \( \frac{\f (t)}{t} \)^6 \) \)
\end{equation}
as $t$ goes to infinity.
\item If a function $\f (t)$ is of small order of $\sqrt{t}$ as $t$ goes to infinity, then, as $t$ goes to infinity, we have the following expansions:
\begin{align}
\label{app_taylor2}
	&\exp\( \frac{-t+\sqrt{t^2-\f (t)^2}}{2} \)
	= 1 - \frac{\f (t)^2}{4t} + O\( \frac{\f (t)^4}{t^2}  \) ,\\
\label{app_taylor3}
	&\exp\(  \frac{\f(t)^2}{4t}+\frac{-t+\sqrt{t^2-\f (t)^2}}{2} \)
	 = 1 + \frac{1}{t}O\( \frac{\f (t)^4}{t^2} \).
\end{align}
\end{itemize}

\subsection{Properties of modified Bessel functions}
In this section, we collect some properties of the modified Bessel functions
\begin{equation}
\label{Bessel}
	I_{\nu}(s)=\sum_{j=0}^\infty \frac{1}{j ! \Gamma (j+\nu+1)}
		\(\frac{s}{2}\)^{2j+\nu}
\end{equation}
used in this paper from [NO]:
\begin{itemize}
\item For a positive constant $a$, we have
\begin{equation}
\label{Bessel_1}
	\int_{-a}^a\frac{e^{s/2}}{\sqrt{a^2-s^2}}ds=\pi I_0\(\frac{a}{2}\).
\end{equation}
\item Direct computation shows the following recursion:
\begin{equation}
\label{Bessel_2}
	I_0' (s)=I_1(s),
	\ I_1'(s)=I_0(s)-\frac{1}{s}I_1(s),
	\ \frac{1}{s} \frac{d}{ds}\( \frac{I_\ell(s)}{s^\ell} \) =\frac{I_{\ell+1}(s)}{s^{\ell +1}}.
\end{equation}
\item The modified Bessel function $I_\nu (s)$ has the expansion
\begin{align}
\nonumber
	I_\nu (s)=\frac{e^s}{\sqrt{2\pi s}}
	&\( 1-\frac{(\nu-1/2)(\nu+1/2)}{2s}
		+\frac{(\nu-1/2)(\nu-3/2)(\nu+3/2)(\nu+1/2)}{2!2^2s^2} \right. \\
	&\quad \left. -\cdots+(-1)^\ell \frac{1}{\ell !2^\ell s^\ell}
		\prod_{j=1}^{\ell} \( \nu-(j-1/2) \) \( \nu+(j-1/2)\) 
	 +O\( \frac{1}{s^{\ell+1}} \) \)
\label{Bessel_3}
\end{align}
as $s$ goes to infinity.
\end{itemize}


\no Shigehiro SAKATA

\no E-mail: sakata@cc.miyazaki-u.ac.jp

\no Address: Faculty of Education and Culture, University of Miyazaki, 1-1 Gakuen Kibana-dai West, Miyazaki city, Miyazaki prefecture, 889-2192, Japan\\

\no Yuta WAKASUGI

\no E-mail: yuta.wakasugi@math.nagoya-u.ac.jp

\no Address: Graduate School of Mathematics, Nagoya University, Furocho, Chikusaku, Nagoya, 464-8602, Japan

\end{document}